\documentclass{amsart}
\usepackage{amssymb}
\usepackage{amsmath}
\usepackage{enumerate}

\newtheorem{thm}{Theorem}
\newtheorem{lem}[thm]{Lemma}
\newtheorem{prop}[thm]{Proposition}
\newtheorem{rem}[thm]{Remark}
\newtheorem{cor}[thm]{Corollary}

\newtheorem{defi}[thm]{Definition}
\newtheorem{quest}[thm]{Question}

\begin{document}

\title{Traces on symmetrically normed operator ideals}

\dedicatory{We dedicate this paper to the memory of Nigel Kalton whose influence on us both has been profound. Without his collaboration this paper would never have been written.}

\author{F. Sukochev}
\address{School of Mathematics and Statistics, University of New South Wales, Sydney, 2052, Australia.} \email{f.sukochev@unsw.edu.au}

\author{D. Zanin}
\address{School of Mathematics and Statistics, University of New South Wales, Sydney, 2052, Australia.} \email{d.zanin@unsw.edu.au}
\date{}


\begin{abstract} For every symmetrically normed ideal $\mathcal{E}$ of compact operators, we give a criterion for the existence of a continuous singular trace on $\mathcal{E}$. We also give a criterion for the existence of a continuous singular trace on $\mathcal{E}$ which respects Hardy-Littlewood majorization. We prove that the class of all continuous singular traces on $\mathcal{E}$ is strictly wider than the class of continuous singular traces which respect Hardy-Littlewood majorization. We establish a canonical bijection between the set of all traces on $\mathcal{E}$ and the set of all symmetric functionals on the corresponding sequence ideal. Similar results are also proved in the setting of semifinite von Neumann algebras.
\end{abstract}

\maketitle

\section{Introduction}\label{introd}

In his groundbreaking paper \cite{Dixmier}, J. Dixmier proved the existence of positive singular traces (that is, linear positive unitarily invariant functionals which vanish on all finite dimensional operators) on the algebra $B(H)$ of all bounded linear operators acting on infinite-dimensional separable Hilbert space $H.$ Namely, if $\psi:\mathbb{R}_+\to\mathbb{R}_+$ is a concave increasing function such that
\begin{equation}\label{simple psi cond}
\lim_{t\to\infty}\frac{\psi(2t)}{\psi(t)}=1,
\end{equation}
then there is a singular trace $\tau_{\omega},$ defined for every positive compact operator $A\in B(H)$ by setting
\begin{equation}\label{dixmier trace}
\tau_{\omega}(A)=\omega(\frac1{\psi(n)}\sum_{k=1}^ns_k(A)).
\end{equation}
Here, $\{s_k(A)\}_{k\in\mathbb{N}}$ is the sequence of singular values of the compact operator $A\in B(H)$ taken in the descending order and $\omega$ is an arbitrary dilation invariant generalised limit on the algebra $l_{\infty}$ of all bounded sequences. This trace is finite on $0\leq A\in B(H)$ if and only if $A$ belongs to the Marcinkiewicz ideal (see e.g. \cite{GohKr1},\cite{GohKr2},\cite{Pietsch})
$$\mathcal{M}_{\psi}:=\{A\in B(H):\ \sup_{n\in\mathbb{N}}\frac1{\psi(n)}\sum_{k=1}^ns_k(A)<\infty\}.$$
In \cite{KSS}, Dixmier's result was extended to an arbitrary Marcinkiewicz ideal $\mathcal{M}_{\psi}$ with the following condition on $\psi$
\begin{equation}\label{dpss psi cond}
\liminf_{t\to\infty}\frac{\psi(2t)}{\psi(t)}=1.
\end{equation}

All the traces defined above by formula \eqref{dixmier trace} vanish on the ideal $\mathcal{L}_1$ consisting of all compact operators $A\in B(H)$ such that $\sum_{k=1}^\infty s_k(A)<\infty$.

An ideal $\mathcal{E}$ of algebra $B(H)$ is said to be symmetrically normed if $\{s_k(B)\}_{k\in\mathbb{N}}\leq\{s_k(A)\}_{k\in\mathbb{N}}$ and $A\in\mathcal{E}$ implies that $\|B\|_{\mathcal{E}}\leq\|A\|_{\mathcal{E}}$ (see \cite{GohKr1}, \cite{GohKr2}, \cite{Simon}\footnote{ We have to caution the reader that in Theorem 1.16 of \cite{Simon} the assertion $(b)$ does not hold for the norm of an arbitrary symmetrically normed ideal $\mathcal {E}$ (see e.g. corresponding counterexamples in \cite[p. 83]{KS}).}, \cite{Schatten}, \cite{KScan}). Since the ideal $\mathcal{M}_{\psi}$ is just a special example of symmetrically normed operator ideal, the following question (suggested in \cite{KSS}, \cite{GI1}, \cite{GI2}, \cite{DPSSS}) arises naturally.

\begin{quest}\label{q1} Which symmetrically normed operator ideals admit a nontrivial singular trace\footnote{
In this paper, we exclusively deal with positive traces}?
\end{quest}

In analyzing Dixmier's proof of the linearity of $\tau_{\omega}$ given by \eqref{simple psi cond}, it was observed in \cite{KSS} (see also \cite{CaSuk}) that $\tau_{\omega}$ possesses the following fundamental property, namely if $0\leq A,B\in\mathcal{M}_{\psi}$ are such that
\begin{equation}\label{hl maj disk}
\sum_{k=1}^ns_k(B)\leq\sum_{k=1}^ns_k(A),\quad\forall n\in\mathbb{N},
\end{equation}
then $\tau_{\omega}(B)\leq\tau_{\omega}(A).$ Such a class of traces was termed \lq\lq fully symmetric\rq\rq in \cite{KScan}, \cite{SZ} (see also earlier papers \cite{DPSS},\cite{LSS}, where the term \lq\lq symmetric\rq\rq was used). It is natural to consider such traces only on fully symmetrically normed operator ideals $\mathcal{E}$ (that is, on symmetrically normed operator ideals $\mathcal{E}$ satisfying the condition: if $A,B$ satisfy \eqref{hl maj disk} and $A\in\mathcal{E},$ then $B\in\mathcal{E}$ and $\|B\|_{\mathcal{E}}\leq\|A\|_{\mathcal{E}}$). In fact, it was established in \cite{DPSS} that every Marcinkiewicz ideal $\mathcal{M}_{\psi}$ with $\psi$ satisfying the condition \eqref{dpss psi cond} possesses fully symmetric traces. Furthermore, in the recent paper \cite{KSS}, the following unexpected result was established. If $\psi$ satisfies the condition \eqref{dpss psi cond}, then every fully symmetric trace on $\mathcal{M}_{\psi}$ is a Dixmier trace $\tau_{\omega}$ for some $\omega.$

The following question ( also suggested in \cite{KSS}, \cite{DPSSS}, \cite{GI1}, \cite{GI2}) arises naturally.

\begin{quest}\label{q2} Which fully symmetrically normed operator ideals admit a nontrivial singular trace which is fully symmetric?
\end{quest}

In papers \cite{GI1},\cite{GI2} the following two problems (closely related to Question \ref{q1} and Question \ref{q2}) were also suggested.

\begin{quest}\label{q3} Which fully symmetrically normed operator ideals admit a trace which is not fully symmetric?
\end{quest}

Let us fix an orthonormal basis $\{e_n\}_{n\in\mathbb{N}}$ in $H.$ An operator $A\in B(H)$ is called diagonal if $(Ae_n,e_m)=0$ for every $n\neq m.$

\begin{quest}\label{q4} Let the mapping $\varphi:\mathcal{E}\to\mathbb{C}$ be unitarily invariant. Suppose that $\varphi$ is linear on the subset of all diagonal operators from $\mathcal{E}.$ Does it imply that $\varphi$ is a trace on $\mathcal{E}?$
\end{quest}

In some very special cases (for principal ideals contained in $\mathcal{L}_1,$ which are, strictly speaking, not symmetrically normed ideals), Question \ref{q3} was answered in the affirmative\footnote{We are grateful to the referee for this remark.} in \cite{Wodzicki}. In \cite{KScan}, question \ref{q3} was answered in the affirmative for the special case of Marcin\-kiewicz ideals under the assumption \eqref{simple psi cond}. It should be pointed out that the method used in \cite{KScan} cannot be extended to an arbitrary Marcinkiewicz ideal $\mathcal{M}_{\psi}$ and, furthermore, cannot be extended to a general symmetrically normed operator ideal. Question \ref{q4} was answered in \cite{KScan} in full generality using deep results from \cite{DK,DFWW} (see also \cite{DFWW_pre}).

The following theorem is the main result of this paper. It yields answers to Questions \ref{q1}--\ref{q3}. In the course of the proof of Theorem \ref{main theorem of the paper}, we also present a new (and very simple) proof answering Question \ref{q4}. Prior to stating Theorem \ref{main theorem of the paper}, we make a few preliminary observations, for which we are grateful to the referee.

Any trace $\varphi:\mathcal{E}\to\mathbb{C}$ obeys the condition
$$\frac1m\varphi(A^{\oplus m})=\varphi(A),\quad A\in\mathcal{E},m\geq1.$$
Here, the direct sum $A^{\oplus m}$ is formed with respect to some arbitrary Hilbert space isomorphism $H^{\oplus m}\simeq H.$ Thus, traces are closely related to the following convex (see Lemma \ref{pietsch funk} below) functional on $\mathcal{E}.$
$$\pi:A\to\lim_{m\to\infty}\frac1m\|A^{\oplus m}\|_{\mathcal{E}},\quad A\in\mathcal{E}.$$
The non-triviality of the functional $\pi:\mathcal{E}\to\mathbb{R}$ is an obvious necessary condition for the existence of a trace.

\begin{thm}\label{main theorem of the paper} Let $\mathcal{E}$ be a symmetrically normed operator ideal. Consider the following conditions.
\begin{enumerate}
\item There exist nontrivial singular traces on $\mathcal{E}.$
\item There exist nontrivial singular traces on $\mathcal{E},$ which are fully symmetric.
\item There exist nontrivial singular traces on $\mathcal{E},$ which are not fully symmetric.
\item $\mathcal{E}\neq\mathcal{L}_1$ and there exist an operator $A\in\mathcal{E}$ such that
\begin{equation}\label{tensor condit}
\lim_{m\to\infty}\frac1m\|A^{\oplus m}\|_{\mathcal{E}}>0.
\end{equation}
\end{enumerate}
\begin{enumerate}[(i)]
\item The conditions $(1)$ and $(4)$ are equivalent for every symmetrically normed operator ideal $\mathcal{E}.$
\item The conditions $(1),$ $(2)$ and $(4)$ are equivalent for every fully symmetrically normed operator ideal $\mathcal{E}.$
\item The conditions $(1)-(4)$ are equivalent for every fully symmetrically normed operator ideal $\mathcal{E}$ equipped with a Fatou norm.
\end{enumerate}
\end{thm}

Recall that the norm on a symmetrically normed operator ideal $\mathcal{E}$ is called a Fatou norm if the unit ball of $\mathcal{E}$ is closed with respect to strong (or, equivalently, weak) operator convergence. Observe that classical ideals (such as Schatten-von Neumann ideals $\mathcal{L}_p,$ Marcinkiewicz, Orlicz and Lorentz ideals \cite{GohKr1}, \cite{GohKr2}, \cite{Simon}) have a Fatou norm. In fact, in some standard references on the subject (e.g. Simon's book \cite{Simon}), the requirement that symmetrically normed operator ideal has a Fatou norm appears to be a part of the definition. Similarly, in the book \cite{LT2}, devoted to the study of symmetric\footnote{termed there \lq\lq rearrangement invariant\rq\rq.} function spaces (which are a commutative counterpart of symmetrically normed operator ideals), an assumption that the norm is a Fatou norm is incorporated into the definition \cite[p. 118]{LT2}.

The proof of Theorem \ref{main theorem of the paper} is given in Section \ref{prmain}. In fact, in this paper we will prove a more general result for symmetric spaces associated with semifinite von Neumann algebras. The precise statements are given in Section \ref{sushestvovanie1} (see Theorems \ref{sim exists}, \ref{sim sing exists inf}, \ref{sim sing exists fin}), Section \ref{sushestvovanie2} (see Theorems \ref{fs exists}, \ref{fs sing exists inf}, \ref{fs sing exists fin}) and Section \ref{ne vpol sim} (see Theorems \ref{nonintegrable theorem}, \ref{integrable theorem}). The appendix contains the proof of important technical results for which we were unable to find a suitable reference. We also present a new and short proof of the Figiel-Kalton theorem from \cite{FK}.

Finally, we say a few words about our proof and its relation to the previous results in the literature. Our strategy is based on the approach from recent papers \cite{SZ} and \cite{KSZ}, where condition \eqref{tensor condit} was connected to the geometry of $\mathcal{E}$ (see also \cite{BM}). The condition \eqref{tensor condit} is easy to verify in concrete situations. For example, the following corollary of Theorem \ref{main theorem of the paper} strengthens the main result of \cite{KScan} and complements earlier results of J. Varga \cite{Varga}.
\begin{cor}\label{marcin cor} Every Marcinkiewicz ideal $\mathcal{M}_{\psi}$ with $\psi$ satisfying the condition \eqref{dpss psi cond} admits a trace which is not fully symmetric.
\end{cor}
Indeed, it is proved in \cite[Proposition 2.3]{AstSuk} that the condition $(4)$ of Theorem \ref{main theorem of the paper} is equivalent to the condition \eqref{dpss psi cond} for the Marcinkiewicz ideal $\mathcal{M}_{\psi}.$ Some examples of symmetrically normed operator ideals, which are not Marcinkiewicz ideals, possessing symmetric traces were presented in \cite{DPSSS}. These results are also an immediate corollary of Theorem \ref{main theorem of the paper}.

For completeness, we note that the assertion $(ii)$ in Theorem \ref{main theorem of the paper} holds for a wider class of relatively fully symmetrically normed operator ideals. The latter class is defined as follows: if $A,B\in\mathcal{E}$ are such that \eqref{hl maj disk} holds, then $\|B\|_{\mathcal{E}}\leq\|A\|_{\mathcal{E}}.$ It coincides with the class of all symmetrically normed subspaces of a fully symmetric operator ideal (see \cite{KS})

\section{Definitions and preliminaries}

The theory of singular traces on symmetric operator ideals rests on some classical analysis which we now review for completeness.

As usual, $L_{\infty}(0,\infty)$ is the set of all bounded Lebesgue measurable functions on the semi-axis equipped with the uniform norm. Given a function $x\in L_{\infty}(0,\infty),$  one defines its decreasing rearrangement $t\to \mu(t,x)$ by the formula (see e.g. \cite{KPS})
$$\mu(t,x)=\inf\{s\geq0:\ m(\{x>s\})\leq t\}.$$
Let $H$ be a Hilbert space and let $B(H)$ be the algebra of all bounded operators on $H$ equipped with the uniform norm.

Let $\mathcal{M}\subset B(H)$ be a semi-finite von Neumann algebra equipped with a fixed faithful and normal semi-finite trace $\tau.$ $\mathcal{M}$ is said to be atomic (see \cite[Definition 5.9]{Takesaki}) if every nonzero projection in $\mathcal{M}$ contains a nonzero minimal projection. $\mathcal{M}$ is said to be atomless if there is no minimal projections in $\mathcal{M}.$

For every $A\in\mathcal{M},$ the generalised singular value function $t\to\mu(t,A)$ is defined by the formula (see e.g. \cite{FackKosaki})
$$\mu(t,A)=\inf\{\|Ap\|:\ \tau(1-p)\leq t\}.$$
If, in particular, $\mathcal{M}=B(H),$ then $\mu(A)$ is a step function and, therefore, can be identified with the sequence of singular numbers of the operators $A$ (the singular values are the eigenvalues of the operator $|A|=(A^*A)^{1/2}$ arranged with multiplicity in decreasing order).

Equivalently, $\mu(A)$ can be defined in terms of the distribution function $d_A$ of $A.$ That is, setting
$$d_A(s)=\tau(E^{|A|}(s,\infty)),\quad s\geq0,$$
we obtain
$$\mu(t,A)=\inf\{s\geq0:\ d_A(s)\leq t\},\quad t>0.$$
Here, $E^{|A|}$ denotes the spectral measure of the operator $|A|.$

Using the Jordan decomposition, every operator $A\in B(H)$ can be uniquely written as
$$A=(\Re A)_+-(\Re A)_-+i(\Im A)_+-i(\Im A)_-.$$
Here, $\Re A:=1/2(A+A^*)$ (respectively, $\Im A:=1/2i(A-A^*)$) for any operator $A\in B(H)$ and $B_+=BE^B(0,\infty)$ ( respectively, $B_−=-BE^B(-\infty,0)$) for any self-adjoint operator $B\in B(H).$ Recall that $\Re A,\Im A\in\mathcal{M}$ for every $A\in\mathcal{M}$ and $B_+,B_-\in\mathcal{M}$ for every self-adjoint $B\in\mathcal{M}.$

Further, we need to recall the important notion of Hardy--Littlewood majorization. Let $A,B\in(L_1+L_{\infty})(\mathcal{M}).$ The operator $B$ is said to be majorized by $A$ and written $B\prec\prec A$ if and only if
$$\int_0^t\mu(s,B)ds\leq\int_0^t\mu(s,A)ds,\quad t\geq0.$$
We have (see \cite{FackKosaki})
$$A+B\prec\prec\mu(A)+\mu(B)\prec\prec 2\sigma_{1/2}\mu(A+B)$$
for every positive operators $A,B\in(L_1+L_{\infty})(\mathcal{M}).$

If $s>0,$ the dilation operator $\sigma_s$ is defined by setting
$$(\sigma_s(x))(t)=x(\frac{t}{s}),\quad t>0$$
in the case of the semi-axis. In the case of the interval $(0,1),$ the operator $\sigma_s$ is defined by
$$(\sigma_sx)(t)=
\begin{cases}
x(t/s),& t\leq\min\{1,s\}\\
0,& s<t\leq1.
\end{cases}$$
Similarly, in the sequence case, we define an operator $\sigma_n$ by setting
$$\sigma_n(a_1,a_2,\cdots)=(\underbrace{a_1,\cdots,a_1}_{\mbox{$n$ times}},\underbrace{a_2,\cdots,a_2}_{\mbox{$n$ times}},\cdots)$$
and an operator $\sigma_{1/2}$ by setting
$$\sigma_{1/2}:(a_1,a_2,a_3,a_4,\cdots)\to(\frac{a_1+a_2}2,\frac{a_3+a_4}2,\cdots).$$

\begin{defi} The Banach space $E(\mathcal{M},\tau)\subset(L_1+L_{\infty})(\mathcal{M})$ is said to be a symmetric operator space if the following conditions hold.
\begin{enumerate}
\item Given $A\in E(\mathcal{M},\tau)$ and $B\in(L_1+L_{\infty})(\mathcal{M})$ with $\mu(B)=\mu(A),$ we have $B\in E(\mathcal{M},\tau)$ and $\|B\|_E=\|A\|_E.$
\item Given $0\leq A\in E(\mathcal{M},\tau)$ and $0\leq B\in(L_1+L_{\infty})(\mathcal{M})$ with $B\leq A,$ we have $B\in E(\mathcal{M},\tau)$ and $\|B\|_E\leq\|A\|_E.$
\end{enumerate}
\end{defi}

The space $E(\mathcal{M},\tau)$ is called fully symmetric if for every $A\in E(\mathcal{M},\tau)$ and every $B\in(L_1+L_{\infty})(\mathcal{M})$ with $B\prec\prec A,$ we have $B\in E(\mathcal{M},\tau)$ and $\|B\|_E\leq\|A\|_E.$

The norm on a symmetric space $E(\mathcal{M},\tau)$ is a Fatou norm if the unit ball of $E(\mathcal{M},\tau)$ is closed with respect to strong (or, equivalently, weak) operator convergence. Every symmetric space equipped with a Fatou norm is necessarily fully symmetric.

A linear functional $\varphi:E(\mathcal{M},\tau)\to\mathbb{C}$ is said to be symmetric if $\varphi(B)=\varphi(A)$ for every positive $A,B\in E(\mathcal{M},\tau)$ such that $\mu(B)=\mu(A).$ A linear functional $\varphi:E(\mathcal{M},\tau)\to\mathbb{C}$ is said to be fully symmetric if $\varphi(B)\leq\varphi(A)$ for every positive $A,B\in E(\mathcal{M},\tau)$ such that $B\prec\prec A.$ Every fully symmetric functional is symmetric and bounded. The converse fails \cite{KScan}.

A functional $\varphi:E(\mathcal{M},\tau)\to\mathbb{C}$ is called singular if $\varphi=0$ on $(L_1\cap L_{\infty})(\mathcal{M}).$ If $E(\mathcal{M},\tau)\not\subset L_1(\mathcal{M}),$ then every symmetric functional is singular.

If $E=E(0,\infty)$ and if $\varphi:E\to\mathbb{R}$ is a symmetric functional, then $s\varphi(x)=\varphi(\sigma_sx)$ for every $x\in E.$ If $E=E(0,1)$ and if $\varphi:E\to\mathbb{R}$ is a singular symmetric functional, then $s\varphi(x)=\varphi(\sigma_sx)$ for every $x=\mu(x)\in E.$

Let $E$ be a fully symmetric Banach space either on the interval $(0,1)$ or on the semi-axis. We need the notion of an expectation operator (see \cite{BM}).

Let $\mathcal{A}=\{A_k\}$ be a (finite or infinite) sequence of disjoint sets of finite measure and denote by $\mathfrak{A}$ the collection of all such sequences. Denote by $A_{\infty}$ the complement of $\cup_kA_k.$

The expectation operator $\mathbf{E}(\cdot|\mathcal{A}):L_1+L_{\infty}\to L_1+L_{\infty}$ is defined by setting
$$\mathbf{E}(x|\mathcal{A})=\sum_k\frac1{m(A_k)}(\int_{A_k}x(s)ds)\chi_{A_k}.$$
Note that we do not require $A_{\infty}$ to have finite measure.

Every expectation operator is a contraction both in $L_1$ and $L_{\infty}.$ Therefore,
$$\mathbf{E}(x|\mathcal{A})\prec\prec x,\quad x\in L_1+L_{\infty}.$$
It follows that $\mathbf{E}(\cdot|\mathcal{A})$ is also contraction in $E.$

It will be convenient to introduce the following notation. If $\mathcal{A}$ is a discrete subset of the semi-axis (i.e. a subset without limit points inside $(0,\infty)$), then the elements of $\mathcal{A}\cup\{0\}$ partition the semi-axis. This partition consists of a (finite or infinite) sequence of sets of finite measure. We identify this partition with the set $\mathcal{A}.$ Elements of $\mathcal{A}$ will be called nodes of the partition $\mathcal{A}.$ The corresponding averaging operator will be denoted by $\mathbf{E}(\cdot|\mathcal{A}).$

Let $E$ be a symmetric Banach space either on the interval $(0,1)$ or on the semi-axis. Define the sets
$$\mathcal{D}_E={\rm Lin}(\{x\in E:\ x=\mu(x)\})=\{\mu(a)-\mu(b),\ a,b\in E\},$$
$$Z_E={\rm Lin}(\{x_1-x_2:\ 0\leq x_1,x_2\in E,\ \mu(x_1)=\mu(x_2)\}).$$
Let $C$ be a Hardy operator defined by setting
$$(Cx)(t)=\frac1t\int_0^tx(s)ds.$$

The following theorem was proved in \cite{FK}. For convenience of the reader, we give a new and simple proof in the appendix.

\begin{thm}\label{fk theorem} Let $E$ be a symmetric space on the semi-axis and let $x\in\mathcal{D}_E.$ We have $x\in Z_E$ if and only if $Cx\in E.$ A similar assertion is also valid for the interval $(0,1)$ provided that $\int_0^1x(s)ds=0.$
\end{thm}

The following uniform submajorization was introduced by Kalton and Sukochev in \cite{KS}.

Let $x,y\in L_1(0,1)$ (or $x,y\in (L_1+L_{\infty})(0,\infty)$). We say that $y\lhd x$ if there exists $m\in\mathbb{N}$ such that
\begin{equation}\label{ks ordering}
\int_{ma}^b\mu(s,y)ds\leq\int_a^b\mu(s,x)ds,\quad\forall ma\leq b.
\end{equation}
Let $x,y\in l_{\infty}.$ We say that $y\lhd x$ if there exists $m\in\mathbb{N}$ such that
\begin{equation}\label{ks ordering0}
\sum_{k=ma+1}^b\mu(k,y)\leq\sum_{k=a+1}^b\mu(k,x)\quad\forall ma+1\leq b.
\end{equation}

The following important theorem was proved in \cite{KS} (see Theorem 5.4 and Theorem 6.3 there).
\begin{thm}\label{ks majorization theorem} Let $x,y\in L_1(0,1)$ or $x,y\in (L_1+L_{\infty})(0,\infty)$ or $x,y\in l_{\infty}$ be such that $y\lhd x.$ For every $\varepsilon>0,$ the function $(1-\varepsilon)y$ belongs to a convex hull of the set $\{z:\ \mu(z)\leq\mu(x)\}.$
\end{thm}

This theorem led to the following fundamental result (see \cite{KS}).

\begin{thm} Let $E=E(0,1)$ (or $E=E(0,\infty)$ or $E=E(\mathbb{N})$) be a symmetric Banach space either on the interval $(0,1)$ or on the semi-axis or on $\mathbb{N}.$ It follows that the corresponding set $E(\mathcal{M},\tau)$ is a symmetric Banach space.
\end{thm}

Also, the uniform submajorization permits us to prove the convexity of the functional $\pi:\mathcal{E}\to\mathbb{R}$ defined in Section \ref{introd}.

\begin{lem}\label{pietsch funk} The functional $\pi:\mathcal{E}\to\mathbb{R}$ is convex on every symmetrically normed operator ideal $\mathcal{E}.$
\end{lem}
\begin{proof} Let $E$ be the corresponding symmetrically normed ideal of $l_{\infty}.$ For every $A,B\in\mathcal{E},$ it follows from Proposition 8.6 of \cite{KS} that $\mu(A+B)\lhd\mu(A)+\mu(B).$ Hence, $\sigma_m\mu(A+B)\lhd\sigma_m(\mu(A)+\mu(B)).$ By Theorem \ref{ks majorization theorem}, we have
$$\|\sigma_m\mu(A+B)\|_E\leq\|\sigma_m(\mu(A)+\mu(B))\|_E\leq\|\sigma_m\mu(A)\|_E+\|\sigma_m\mu(B)\|_E.$$
Note that $\|A^{\oplus m}\|_{\mathcal{E}}=\|\sigma_m\mu(A)\|_E.$ Dividing by $m$ and letting $m\to\infty,$ we obtain $$\pi(A+B)\leq\pi(A)+\pi(B).$$
\end{proof}

\section{Lifting of symmetric functionals}

In this section, we explain a canonical bijection between symmetric functionals and traces. In what follows, we require that a semifinite von Neumann algebra $\mathcal{M}$ be either atomless or atomic with traces of all atoms being $1.$

For an atomless von Neumann algebra $\mathcal{M},$ we have (see e.g. \cite{FackKosaki})
$$\int_0^t\mu(s,A)ds=\sup\{\tau(p|A|):\ p\in P(\mathcal{M}),\ \tau(p)=t\},\quad A\in\mathcal{M}.$$
For a atomic von Neumann algebra $\mathcal{M},$ we have (see e.g. \cite{FackKosaki})
$$\sum_{k=1}^m\mu(k,A)=\sup\{\tau(p|A|):\ p\in P(\mathcal{M}),\ \tau(p)=m\},\quad A\in\mathcal{M}.$$
In either case, this implies a remarkable inequality (see e.g. \cite{FackKosaki})
\begin{equation}\label{rear sum estimate}
\mu(A+B)\prec\prec\mu(A)+\mu(B)\prec\prec2\sigma_{1/2}\mu(A+B),\quad 0\leq A,B\in(L_1+L_{\infty})(\mathcal{M}).
\end{equation}

\begin{lem}\label{fi ks major} Let $E=E(0,1)$ (or $E=E(0,\infty)$ or $E=E(\mathbb{N})$) be a symmetric Banach space either on the interval $(0,1)$ or on the semi-axis or on $\mathbb{N}.$ If $x,y\in E_+$ are such that $y\lhd x,$ then $\varphi(y)\leq\varphi(x)$  for every positive symmetric functional $\varphi$ on $E.$
\end{lem}
\begin{proof} Fix $\varepsilon>0.$ By Theorem \ref{ks majorization theorem}, there exist $z_k\in E,$ $1\leq k\leq n,$ and positive numbers $\lambda_k,$ $1\leq k\leq n,$ such that $\mu(z_k)\leq\mu(x)$ for every $1\leq k\leq n$ and
$$(1-\varepsilon)y=\sum_{k=1}^n\lambda_kz_k,\quad \sum_{k=1}^n\lambda_k=1.$$
Since $\varphi$ is positive and symmetric, it follows that
$$\varphi(z_k)\leq\varphi(|z_k|)=\varphi(\mu(z_k))\leq\varphi(\mu(x))=\varphi(x).$$
Therefore, $(1-\varepsilon)\varphi(y)\leq\varphi(x).$ Since $\varepsilon>0$ is arbitrarily small, the assertion follows.
\end{proof}

The following assertion is essentially known. However, we provide the full proof for readers convenience.

\begin{lem}\label{mu sum estimate} Let $\mathcal{M}$ be a semifinite atomless von Neumann algebra and let $A,B\in(L_1+L_{\infty})(\mathcal{M},\tau)$ be positive operators.
$$\int_{2a}^b\mu(s,A+B)ds\leq\int_a^b(\mu(s,A)+\mu(s,B))ds,\quad\forall 2a\leq b,$$
$$\int_{2a}^b(\mu(s,A)+\mu(s,B))ds\leq\int_{2a}^{2b}\mu(s,A+B)ds,\quad\forall 2a\leq b.$$
Similar assertion is valid for atomic von Neumann algebra $\mathcal{M}.$
\end{lem}
\begin{proof} Applying inequality \eqref{rear sum estimate} to the operators $A,B,$ we obtain that
$$\int_0^b\mu(s,A+B)ds\leq\int_0^b(\mu(s,A)+\mu(s,B))ds$$
and
$$\int_0^{2a}\mu(s,A+B)ds\geq\int_0^a(\mu(s,A)+\mu(s,B))ds.$$
Subtracting this inequalities, we obtain
$$\int_{2a}^b\mu(s,A+B)ds\leq\int_a^b(\mu(s,A)+\mu(s,B))ds.$$
Proof of the second inequality is identical.
\end{proof}

The following theorem answers Question \ref{q4} in the affirmative, as also does \cite[Theorem 5.2]{KScan}. The proof below is very simple and based on a completely different approach.

\begin{thm}\label{lifting theorem} Let $E=E(0,1)$ (or $E=E(0,\infty)$ or $E=E(\mathbb{N})$) be a symmetric Banach space either on the interval $(0,1)$ or on the semi-axis or on $\mathbb{N}$ and let $E(\mathcal{M},\tau)$ be the corresponding symmetric Banach operator space.
\begin{enumerate}
\item\label{lifting func} If $\varphi$ is a positive symmetric functional on $E,$ then there exists a positive symmetric functional $\mathcal{L}(\varphi)$ on $E(\mathcal{M},\tau)$ such that $\varphi(x)=\mathcal{L}(\varphi)(A)$ for all positive $x\in E$ and $A\in E(\mathcal{M},\tau)$ such that $\mu(A)=\mu(x).$
\item\label{dropping func} If $\varphi$ is a positive symmetric functional on $E(\mathcal{M},\tau),$ then there exists a positive symmetric functional $\mathcal{L}^{-1}(\varphi)$ on $E$ such that $\varphi(A)=\mathcal{L}^{-1}(\varphi)(x)$ for all positive $x\in E$ and $A\in E(\mathcal{M},\tau)$ such that $\mu(A)=\mu(x).$
\end{enumerate}
\end{thm}
\begin{proof} We will only prove \eqref{lifting func}. Proof of \eqref{dropping func} is identical.

Let $A,B\in E_+(\mathcal{M},\tau).$ It follows from Lemma \ref{mu sum estimate} that
$$\mu(A+B)\lhd\mu(A)+\mu(B)\lhd 2\sigma_{1/2}\mu(A+B).$$
It follows from Lemma \ref{fi ks major} that
$$\varphi(\mu(A+B))\leq\varphi(\mu(A)+\mu(B))\leq\varphi(2\sigma_{1/2}\mu(A+B))=\varphi(\mu(A+B)).$$
It follows that $\mathcal{L}(\varphi)$ is additive on $E_+(\mathcal{M},\tau).$ We than extend it to $E(\mathcal{M},\tau)$ by linearity.
\end{proof}

Theorem \ref{lifting theorem} provides a very natural bijection between the set of all symmetric functionals on $E$ and that on $E(\mathcal{M},\tau),$ observed first for the case of fully symmetric functionals in \cite{DPSS}. Next corollary follows immediately.

\begin{cor} Let $E$ and $E(\mathcal{M},\tau)$ be as in Theorem \ref{lifting theorem}. The functional $\varphi$ is fully symmetric on $E$ if and only if $\mathcal{L}(\varphi)$ is a fully symmetric functional on $E(\mathcal{M},\tau).$
\end{cor}

We also need a lifting between sequence and function spaces. The following space was introduced in \cite{KSZ}.

Let $\mathcal{A}=\{[n-1,n]\}_{n\in\mathbb{N}}$ be a partition of the semi-axis. Clearly, $\mathbf{E}(\cdot|\mathcal{A})$ maps $L_1+L_{\infty}$ into the set of step functions which can be identified with sequences.

\begin{prop}\label{func-seq} Let $E$ be a symmetric Banach sequence space and let $F$ be the linear space of all such  functions $x\in L_{\infty}$ for which $\mathbf{E}(\mu(x)|\mathcal{A})\in E.$ The space $F$ equipped with the norm
$$\|x\|_F=\|x\|_{\infty}+\|\mathbf{E}(\mu(x)|\mathcal{A})\|_E$$
is a symmetric Banach function space.
\end{prop}
The fact that the space $F$ is a Banach space is non-trivial. Proof of this fact was missing in both \cite{KS} and \cite{KSZ}. We include it in the appendix.

Below, we assume that $E$ is embedded into $F.$

\begin{thm}\label{second lifting theorem} Let $E=E(\mathbb{N})$ be a symmetric Banach sequence space and let $F$ be the corresponding function space.
\begin{enumerate}
\item\label{lifting func1} If $\varphi$ is a positive symmetric functional on $E,$ then there exists a positive symmetric functional $\mathcal{L}(\varphi)$ on $F$ such that $\varphi(\mathbf{E}(\mu(x)|\mathcal{A}))=\mathcal{L}(\varphi)(x)$ for all positive $x\in F.$
\item\label{dropping func1} If $\varphi$ is a positive symmetric functional on $F,$ then its restriction on $E$ is a positive symmetric functional. This restriction is an inverse operation for the $\mathcal{L}$ in \eqref{lifting func1}.
\end{enumerate}
\end{thm}
\begin{proof} Let us prove \eqref{lifting func1} $$\varphi(\sigma_{1/2}a)=1/2\varphi(a_1,a_3,\cdots)+1/2\varphi(a_2,a_4,\cdots)=$$
$$=1/2\varphi(a_1,0,a_2,0,\cdots)+1/2\varphi(0,a_2,0,a_4,\cdots)=1/2\varphi(a)$$
for every $a\in E.$

Let $x,y\in F$ be positive. It follows from Lemma \ref{pave maj} that
$$\mathbf{E}(\mu(x+y)|\mathcal{A})\lhd\mathbf{E}(\mu(x)+\mu(y)|\mathcal{A})\lhd 2\sigma_{1/2}\mathbf{E}(\mu(x+y)|\mathcal{A}).$$
It follows from Lemma \ref{fi ks major} that
$$\varphi(\mathbf{E}(\mu(x+y)|\mathcal{A}))=\varphi(\mathbf{E}(\mu(x)+\mu(y)|\mathcal{A}))$$
and \eqref{lifting func1} follows.

The first assertion of \eqref{dropping func1} is trivial. Clearly, $\mu(x)-\mathbf{E}(\mu(x)|\mathcal{A})\in(L_1\cap L_{\infty})(0,\infty).$ If $E\neq l_1,$ then $\varphi(y)=0$ for every $y\in(L_1\cap L_{\infty})(0,\infty)$ and every symmetric functional $\varphi$ on $F.$ If $E=l_1,$ then $F=(L_1\cap L_{\infty})(0,\infty)$ and the only symmetric functional on both spaces is an integral. The second assertion of \eqref{dropping func1} follows.
\end{proof}

\section{Existence of symmetric functionals}\label{sushestvovanie1}

In this section, we present results concerning existence of symmetric functionals on symmetric function spaces. The main results of this section are Theorem \ref{sim exists}, Theorem \ref{sim sing exists inf} and Theorem \ref{sim sing exists fin}.

We need the following variation of the Hahn-Banach theorem.

\begin{lem}\label{positive hahn-banach} Let $E$ be a partially ordered linear space and let $p:E\to\mathbb{R}$ be convex and monotone functional. For every $x_0\in E,$ there exists a positive linear functional $\varphi:E\to\mathbb{R}$ such that $\varphi\leq p$ and $\varphi(x_0)=p(x_0).$
\end{lem}
\begin{proof} The existence of $\varphi$ follows from the Hahn-Banach theorem. We only have to prove that $\varphi\geq0.$ If $z\geq0,$ then $\varphi(x_0-z)\leq p(x_0-z).$ Therefore,
$$\varphi(z)\geq\varphi(x_0)-p(x_0-z)=p(x_0)-p(x_0-z)\geq0$$
due to the fact that $z\geq0$ and $p$ is monotone.
\end{proof}

Define operators $M_m:(L_1+L_{\infty})(0,\infty)\to(L_1+L_{\infty})(0,\infty)$ (or, $M_m:L_1(0,1)\to L_1(0,1)$) by setting
$$(M_mx)(t)=\frac1{t\log(m)}\int_{t/m}^tx(s),\quad m\geq2.$$

\begin{lem}\label{mn estimate} If $0\leq x\in L_1+L_{\infty}$ (or, $0\leq x\in L_1(0,1)$), then
$$\int_a^{b/m}x(s)ds\leq\int_a^b(M_mx)(s)ds\leq\int_{a/m}^bx(s)ds$$
provided that $ma\leq b.$ In particular, $m^{-1}\sigma_mx\lhd M_mx\lhd x$ provided that $x=\mu(x).$
\end{lem}
\begin{proof} Clearly,
$$\int_a^b(M_mx)(s)ds=\frac1{\log(m)}\int_a^b\int_{t/m}^tx(s)ds\frac{dt}{t}=$$
$$=\frac1{\log(m)}\int_{a/m}^b\int_{\max\{a,s\}}^{\min\{ms,b\}}\frac{dt}{t}x(s)ds=\frac1{\log(m)}\int_{a/m}^bx(s)\log(\frac{\min\{ms,b\}}{\max\{a,s\}})ds.$$
The integrand does not exceed $x(s)\log(m)$ and the second inequality follows immediately. The integrand is positive and is equal to $x(s)\log(m)$ for $s\in(a,b/m).$ The first inequality follows.
\end{proof}

\begin{cor}\label{mn contraction corollary} If $E$ is a symmetric Banach function space either on the interval $(0,1)$ or on the semi-axis, then $M_m:E\to E$ is a contraction for $m\in\mathbb{N}.$
\end{cor}
\begin{proof} Let $x=\mu(x)\in E.$ It follows from Lemma \ref{mn estimate} that $M_mx\lhd x.$ It follows from theorem \ref{ks majorization theorem} that, for every $\varepsilon>0,$ the function $(1-\varepsilon)M_mx$ belongs to a convex hull of the set $\{z:\ \mu(z)\leq\mu(x)\}.$ Therefore, $M_mx\in E$ and $(1-\varepsilon)\|M_mx\|_E\leq\|x\|_E.$ Since $\varepsilon$ is arbitrarily small, the assertion follows.
\end{proof}

\begin{lem}\label{p extension} Let $E$ be a symmetric Banach space either on the interval $(0,1)$ or on the semi-axis. Let $p:\mathcal{D}_E\to\mathbb{R}$ be convex and monotone functional. If $p=0$ on $Z_E\cap\mathcal{D}_E,$ then $p$ extends to a convex monotone functional $p:E\to\mathbb{R}$ by setting
$$p(x)=p(\mu(x_+)-\mu(x_-)).$$
Also, $p(x)=0$ for every $x\in Z_E.$
\end{lem}
\begin{proof} If $x\in\mathcal{D}_E,$ then $x-\mu(x_+)+\mu(x_-)\in Z_E\cap\mathcal{D}_E.$ Therefore, $p(x-\mu(x_+)+\mu(x_-))=0$ and, due to the convexity of $p,$ $p(x)=p(\mu(x_+)-\mu(x_-)).$ This proves the correctness of the definition.

For $x,y\in E,$ we have
$$\mu((x+y)_+)-\mu((x+y)_-)-\mu(x_+)+\mu(x_-)-\mu(y_+)+\mu(y_-)\in Z_E\cap\mathcal{D}_E.$$
It follows that
$$p(\mu((x+y)_+)-\mu((x+y)_-)-\mu(x_+)+\mu(x_-)-\mu(y_+)+\mu(y_-))=0$$
and
$$p(x+y)=p(\mu((x+y)_+)-\mu((x+y)_-))=$$
$$=p(\mu(x_+)-\mu(x_-)+\mu(y_+)-\mu(y_-))\leq p(x)+p(y).$$

Since $p$ is monotone on $\mathcal{D}_E,$ then $p(y)\leq0$ for every $0\geq y\in\mathcal{D}_E.$ It follows that $p(y)=p(-\mu(y))\leq 0$ for $0\geq y\in E.$ Therefore, $p(x+y)\leq p(x)+p(y)\leq p(x)$ for every $0\geq y\in E.$ 
\end{proof}

\begin{lem}\label{required p} Let $E$ be a symmetric Banach space either on the interval $(0,1)$ or on the semi-axis. The functional
$$p:x\to\limsup_{m\to\infty}\|(M_mx)_+\|_E,\quad x\in\mathcal{D}_E$$
satisfies the assumptions of Lemma \ref{p extension}. Also, for every $x\in\mathcal{D}_E,$ we have $p(x)\leq\|x\|_E.$
\end{lem}
\begin{proof} It follows from Corollary \ref{mn contraction corollary} that
$$\|(M_mx)_+\|_E\leq\|M_mx\|_E\leq\|x\|_E,\quad x\in E.$$
It follows that
$$p(x)=\limsup_{m\to\infty}\|(M_mx)_+\|_E\leq\|x\|_E,\quad x\in\mathcal{D}_E.$$
Clearly, the mappings $x\to(M_mx)_+$ are convex and monotone. So are the mappings $x\to\|(M_mx)_+\|_E.$ Therefore, $p:\mathcal{D}_E\to\mathbb{R}$ is a convex and monotone functional.

If $x\in Z_E\cap\mathcal{D}_E,$ then by Theorem \ref{fk theorem} $|Cx|\in E.$ Therefore,
$$(M_mx)(t)\leq\frac1{\log(m)}(|\frac1t\int_0^{t/m}x(s)ds|+|\frac1t\int_0^tx(s)ds|)\leq$$
$$\leq\frac1{\log(m)}(\frac1m\sigma_m|Cx|+|Cx|)(t).$$
Since $\|\sigma_m\|_{E\to E}\leq m$ (see \cite[Theorem II.4.5]{KPS}), it follows that
$$\|(M_mx)_+\|_E\leq\frac2{\log(m)}\|Cx\|_E$$
and $p(x)=0.$ 
\end{proof}

\begin{thm}\label{sim exists} Let $E=E(0,\infty)$ be a symmetric Banach space on the semi-axis. For a given $0\leq x\in E,$ there exists a symmetric linear functional $\varphi:E\to\mathbb{R}$ such that
$$\varphi(x)=\lim_{m\to\infty}\frac1m\|\sigma_m(\mu(x))\|_E.$$
\end{thm}
\begin{proof} Without loss of generality, $x=\mu(x).$ Let $p$ be the convex monotone functional constructed in Lemma \ref{required p}. It follows from Lemma \ref{positive hahn-banach} that there exist a positive linear functional $\varphi$ on $E$ such that $\varphi\leq p$ and $\varphi(x)=p(x).$ Since $p(z)=0$ for every $z\in Z_E,$ it follows that $\varphi(z)=0$ for every $z\in Z_E.$ Therefore, $\varphi$ is a symmetric functional.

Since $\varphi(z)\leq p(z)\leq\|z\|_E$ for every $z=\mu(z)\in E,$ it follows that $\|\varphi\|_{E^*}\leq 1.$ Therefore,
$$\varphi(x)=\varphi(\frac1m\sigma_mx)\leq\frac1m\|\sigma_mx\|_E.$$
Passing $m\to\infty,$ we obtain
$$\varphi(x)\leq\lim_{m\to\infty}\frac1m\|\sigma_m\mu(x)\|_E.$$
On the other hand, It follows from Lemma \ref{mn estimate} that $m^{-1}\sigma_mx\lhd M_mx.$ Therefore,
$$p(x)=\limsup_{m\to\infty}\|M_mx\|_E\geq\lim_{m\to\infty}\frac1m\|\sigma_m\mu(x)\|_E.$$
The assertion follows immediately. 
\end{proof}

Consider the functional $\pi:E\to E$ (identical to the one defined in Section \ref{introd}).
\begin{equation}\label{pi deff}
\pi(x)=\lim_{m\to\infty}\frac1m\|x^{\oplus m}\|_E,\quad x\in E.
\end{equation}
Note that $\pi(-x)=\pi(x)$ for every $x\in E.$ If $p$ is a functionals defined in Lemma \ref{required p}, then $p(-x)=0$ for positive $x\in E.$ Therefore, $p\neq\pi.$ However, the assertion below follows from Theorem \ref{sim exists}.

\begin{lem}\label{pietsch request1} Let $E=E(0,\infty)$ be a symmetric Banach space on the semi-axis. Let $p$ and $\pi$ be the convex functionals on $E$ defined in Lemma \ref{required p} and \eqref{pi deff}, respectively. For every positive $x\in E,$ we have $p(x)=\pi(x).$
\end{lem}
\begin{proof} For every $x\in E,$ consider the functional $\varphi$ constructed in Theorem \ref{sim exists}. By construction, we have $\varphi(x)=p(x)=\pi(x).$
\end{proof}

If $E\not\subset L_1(0,\infty),$ then the functional $\varphi$ constructed in Theorem \ref{sim exists} is necessarily singular. The case $E\subset L_1$ requires more detailed treatment.

\begin{lem}\label{weak limit lemma} Let $E$ be a symmetric (respectively, fully symmetric) Banach function space either on the interval $(0,1)$ or on the semi-axis. Let $\{\varphi_i\}_{i\in\mathbb{I}}\in E^*$ be a net and let $\varphi\in E^*$ be such that $\varphi_i\to\varphi$ $*-$weakly.
\begin{enumerate}
\item If every $\varphi_i$ is symmetric, then $\varphi$ is symmetric.
\item If every $\varphi_i$ is fully symmetric, then $\varphi$ is fully symmetric.
\end{enumerate}
\end{lem}
\begin{proof} Let each $\varphi_i$ be symmetric. If $0\leq x_1,x_2\in E$ are such that $\mu(x_1)=\mu(x_2),$ then
$$\varphi(x_1)=\lim_{i\in\mathbb{I}}\varphi_i(x_1)=\lim_{i\in\mathbb{I}}\varphi_i(x_2)=\varphi(x_2).$$
Hence, $\varphi$ is symmetric.

Let each $\varphi_i$ be fully symmetric. Thus, $\varphi_i(x)\leq0$ for every $x\in\mathcal{D}_E$ such that $Cx\leq0.$ Therefore, $\varphi(x)=\lim_{i\in\mathbb{I}}\varphi_i(x)\leq0$ for every $x\in\mathcal{D}_E$ such that $Cx\leq0.$

Let $x_1,x_2\in E$ be positive elements such that $x_1\prec\prec x_2.$ Therefore, $z=\mu(x_1)-\mu(x_2)\in\mathcal{D}_E$ and $Cz\leq0.$ It follows from above that $\varphi(z)\leq 0.$ Hence, $\varphi$ is a fully symmetric functional. 
\end{proof}

\begin{lem}\label{sing constr} Let $E$ be a symmetric (respectively, fully symmetric) Banach function space either on the interval $(0,1)$ or on the semi-axis and let $\varphi$ be a symmetric (respectively, fully symmetric) functional on $E.$ The formula
$$\varphi_{sing}(x)=\lim_{n\to\infty}\varphi(\mu(x)\chi_{(0,1/n)}),\quad 0\leq x\in E.$$
defines a singular symmetric (respectivley, fully symmetric) linear functional on $E.$
\end{lem}
\begin{proof} If $x,y\in E$ are positive functions, then
$$\mu(x+y)\chi_{(0,1/n)}\lhd(\mu(x)+\mu(y))\chi_{(0,1/n)}\lhd 2\sigma_{1/2}\mu(x+y)\chi_{(0,1/n)}.$$
Taking the limit as $n\to\infty,$ we derive from Lemma \ref{fi ks major} that
$$\varphi_{sing}(\mu(x+y))=\varphi_{sing}(\mu(x)+\mu(y)).$$
Since $\varphi$ is symmetric, it follows that
$$\varphi_{sing}(x+y)=\varphi_{sing}(\mu(x+y))=\varphi_{sing}(\mu(x)+\mu(y))=\varphi_{sing}(x)+\varphi_{sing}(y).$$
Hence, $\varphi_{sing}$ is an additive functional on $E_+.$ Therefore, it extends to a linear functional on $E.$ Clearly, $\varphi_{sing}$ is symmetric. Second assertion is trivial. 
\end{proof}

In fact, the construction in Lemma \ref{sing constr} gives a singular part of the functional $\varphi$ as defined by Yosida-Hewitt theorem.

\begin{lem}\label{normal construction} Let $E=E(0,\infty)\subset L_1(0,\infty)$ be a symmetric Banach function space on the semi-axis and let $\varphi$ be a symmetric functional on $E.$ If $\varphi_{sing}$ is a functional constructed in Lemma \ref{sing constr}, then $\varphi-\varphi_{sing}$ is a normal functional (that is, an integral).
\end{lem}
\begin{proof} It is clear that
$$0\leq\varphi_{sing}(z)\leq\|z\|_{\infty}\lim_{n\to\infty}\varphi(\chi_{(0,1/n)})=0$$
for every positive $z\in(L_1\cap L_{\infty})(0,\infty).$ It follows that $\varphi_{sing}(\mu(x)\chi_{(1/n,\infty)})=0$ for every $x\in E.$ Therefore,
\begin{equation}\label{l1lin predel}
(\varphi-\varphi_{sing})(x)=\lim_{n\to\infty}\varphi(\mu(x)\chi_{(1/n,\infty)})=\lim_{n\to\infty}(\varphi-\varphi_{sing})(\mu(x)\chi_{(1/n,\infty)}).
\end{equation}
On the other hand, for every positive $z\in(L_1\cap L_{\infty})(0,\infty)$ with $\|z\|_{\infty}=1,$ we have $z\prec\chi_{(0,\|z\|_1)}.$ It is proved in \cite[Theorem 23]{SZ} that $z$ belongs to the closure (in the topology of $L_1\cap L_{\infty}$) of the set $\{u\geq0:\ \mu(u)=\chi_{(0,\|z\|_1)}\}.$ Thus,
$$(\varphi-\varphi_{sing})(z)=(\varphi-\varphi_{sing})(\chi_{(0,\|z\|_1)})=\|z\|_1(\varphi-\varphi_{sing})(\chi_{(0,1)}).$$
By linearity,
\begin{equation}\label{l1lin case}
(\varphi-\varphi_{sing})(z)=(\varphi-\varphi_{sing})(\chi_{(0,1)})\cdot\int_0^{\infty}z(s)ds,\quad\forall z\in(L_1\cap L_{\infty})(0,\infty).
\end{equation}
It follows that from \eqref{l1lin predel} and \eqref{l1lin case} that
$$(\varphi-\varphi_{sing})(x)=\lim_{n\to\infty}\int_{1/n}^{\infty}\mu(s,x)ds\cdot(\varphi-\varphi_{sing})(0,1)=\int_0^1x(s)ds\cdot(\varphi-\varphi_{sing})(\chi_{(0,1)})$$
for every positive function $x\in E.$ The assertion follows immediately.
\end{proof}

\begin{thm}\label{sim sing exists inf} Let $E\subset L_1(0,\infty)$ be a symmetric Banach space on the semi-axis. For a given $0\leq x\in E,$ there exists a singular symmetric linear functional $\varphi_{sing}$ such that
$$\varphi_{sing}(x)=\lim_{m\to\infty}\frac1m\|\sigma_m(\mu(x))\chi_{(0,1)}\|_E.$$
\end{thm}
\begin{proof} Apply Theorem \ref{sim exists} to the function $\mu(x)\chi_{(0,1/n)}.$ It follows that there exists a  symmetric linear functional $\varphi_n$ such that $\|\varphi_n\|_{E^*}\leq1$ and
$$\varphi_n(\mu(x)\chi_{(0,1/n)})=\lim_{m\to\infty}\frac1m\|\sigma_m(\mu(x)\chi_{(0,1/n)})\|_E\geq\lim_{m\to\infty}\frac1m\|\sigma_m(\mu(x))\chi_{(0,1)}\|_E.$$
Since the unit ball in $E^*$ is $*-$weakly compact (Banach-Alaoglu theorem), there exists a convergent subnet $\psi_i=\varphi_{F(i)},$ $i\in\mathbb{I},$ of the sequence $\varphi_n,$ $n\in\mathbb{N}.$ Let $\psi_i\to\varphi.$ It follows from Lemma \ref{weak limit lemma} that $\varphi$ is a symmetric functional.

By the definition of a subnet (see \cite[Section IV.2]{ReedSimon}), for every fixed $n\in\mathbb{N},$ there exists $i_n\in\mathbb{I}$ such that $F(i)>n$ for every $i>i_n.$ Thus, for every $i>i_n,$ we have
$$\psi_i(\mu(x)\chi_{(0,1/n)})\geq\varphi_{F(i)}(\mu(x)\chi_{(0,1/F(i))})\geq\lim_{m\to\infty}\frac1m\|\sigma_m(\mu(x))\chi_{(0,1)}\|_E.$$
The subnet $\psi_i,$ $i_n<i\in\mathbb{I}$ converges to the same limit $\varphi.$ Therefore,
$$\varphi(\mu(x)\chi_{(0,1/n)})\geq\lim_{m\to\infty}\frac1m\|\sigma_m(\mu(x))\chi_{(0,1)}\|_E.$$
Now, taking the limit as $n\to\infty,$ we obtain the inequality
$$\varphi_{sing}(x)\geq\lim_{m\to\infty}\frac1m\|\sigma_m(\mu(x))\chi_{(0,1)}\|_E,$$
where $\varphi_{sing}$ is a singular symmetric functional defined in Lemma \ref{sing constr}. The opposite inequality is trivial.
\end{proof}

\begin{thm}\label{sim sing exists fin} Let $E$ be a symmetric Banach space on the interval $(0,1).$ For a given $0\leq x\in E,$ there exists a singular symmetric linear functional $\varphi_{sing}$ such that
$$\varphi_{sing}(x)=\lim_{m\to\infty}\frac1m\|\sigma_m(\mu(x))\|_E.$$
\end{thm}
\begin{proof} Let $F$ be a symmetric Banach space on the semi-axis with a norm given by the formula
$$\|x\|_F=\|\mu(x)\chi_{(0,1)}\|_E+\|x\|_1,\quad \forall x\in F.$$
Clearly, $F\subset L_1(0,\infty).$ Applying Theorem \ref{sim sing exists inf}, we obtain a symmetric singular functional $\varphi$ on $F$ such that
$$\varphi(x)=\lim_{m\to\infty}\frac1m\|\sigma_m(\mu(x))\chi_{(0,1)}\|_F=\lim_{m\to\infty}\frac1m\|\sigma_m(\mu(x))\|_E.$$

\end{proof}

\section{Existence of fully symmetric functionals}\label{sushestvovanie2}

In this section, we present results concerning existence of fully symmetric functionals on fully symmetric function spaces. The main results of this section are Theorem \ref{fs exists}, Theorem \ref{fs sing exists inf} and Theorem \ref{fs sing exists fin}.

\begin{lem}\label{hmn monotone} Let $E$ be a symmetric Banach function space either on the interval $(0,1)$ or on the semi-axis. If $x,z\in\mathcal{D}_E$ are such that $Cx\leq Cz,$ then $CM_mx\leq CM_mz.$
\end{lem}
\begin{proof} Let $x=\mu(a)-\mu(b)$ and $z=\mu(c)-\mu(d)$ with $a,b,c,d\in E.$ It follows from assumption $Cx\leq Cz$ that $C(\mu(a)+\mu(d))\leq C(\mu(b)+\mu(c))$ or, equivalently, $\mu(a)+\mu(d)\prec\prec\mu(b)+\mu(c).$

Arguing as in Lemma \ref{mn estimate}, we have
$$\int_0^t(M_mz)(s)ds=\int_0^tz(s)h(s,t)ds$$
with
$$h(s,t)=\left\{
\begin{aligned}
1,\quad 0\leq s\leq t/m\\
\frac{\log(t/s)}{\log(m)},\quad  t/m\leq s\leq t
\end{aligned}
\right.$$
It is now clear that
$$\int_0^tM_m(\mu(a)+\mu(d))(s)ds=\int_0^t(\mu(s,a)+\mu(s,d))h(s,t)ds,$$
$$\int_0^tM_m(\mu(b)+\mu(c))(s)ds=\int_0^t(\mu(s,b)+\mu(s,c))h(s,t)ds.$$
Clearly, $h$ is positive and decreasing with respect to $s.$ It follows from \cite[Equality 2.36]{KPS} that
$$M_m(\mu(a)+\mu(d))\prec\prec M_m(\mu(b)+\mu(c))$$
and the assertion follows.
\end{proof}

\begin{lem}\label{ph monotone} Let $E$ be a fully symmetric Banach function space either on the interval $(0,1)$ or on the semi-axis and let $x=\mu(x)\in E.$ If $z\in\mathcal{D}_E$ is such that $Cx\leq Cz,$ then $p(x)\leq p(z).$
\end{lem}
\begin{proof} Since $M_mx$ is decreasing, it follows from Lemma \ref{hmn monotone} that
$$\int_0^t\mu(s,M_mx)ds=\int_0^t(M_mx)(s)ds\leq\int_0^t(M_mz)_+(s)ds\leq\int_0^t\mu(s,(M_mz)_+)ds.$$
Therefore, $(M_mx)_+=M_mx\prec\prec(M_mz)_+.$ The assertion follows now from the definition of the functional $p.$ 
\end{proof}

\begin{lem}\label{required q} Let $E$ be a fully symmetric Banach function space either on the interval $(0,1)$ or on the semi-axis. Let $p$ be the functional constructed in Lemma \ref{required p}. The functional
$$q(x)=\inf\{p(z):\ z\in\mathcal{D}_E,\ Cx\leq Cz\},\quad x\in\mathcal{D}_E$$
satisfies the assumptions of Lemma \ref{p extension}.
\end{lem}
\begin{proof} It is clear from the definition of $q$ that $q\leq p$ and that $q$ is a positive functional.

We claim that $q$ is convex on $\mathcal{D}_E.$ Let $x_1,x_2\in\mathcal{D}_E.$ Fix $\varepsilon>0$ and select $z_1,z_2\in\mathcal{D}_E$ such that $Cx_i\leq Cz_i$ and $p(z_i)\leq q(x_i)+\varepsilon$ for $i=1,2.$ Thus, $C(x_1+x_2)\leq C(z_1+z_2)$ and
$$q(x_1+x_2)\leq p(z_1+z_2)\leq p(z_1)+p(z_2)\leq q(x_1)+q(x_2)+2\varepsilon.$$
Since $\varepsilon$ is arbitrarily small, the claim follows.

We claim that $q$ is monotone on $\mathcal{D}_E.$ Let $x_1,x_2\in\mathcal{D}_E$ be such that $x_1\leq x_2.$ Fix $\varepsilon>0$ and select $z\in\mathcal{D}_E$ such that $Cx_2\leq Cz$ and $p(z)\leq q(x_2)+\varepsilon.$ Thus, $Cx_1\leq Cx_2\leq Cz$ and $q(x_1)\leq p(z)\leq q(x_2)+\varepsilon.$ Since $\varepsilon$ is arbitrarily small, the claim follows.

For $x\in Z_E\cap\mathcal{D}_E,$ we have $0\leq q(x)\leq p(x)=0$ and, therefore, $q(x)=0.$ s
\end{proof}

The following theorem is the first main result of this section.

\begin{thm}\label{fs exists} Let $E=E(0,\infty)$ be a fully symmetric Banach space on the semi-axis. For a given $0\leq x\in E,$ there exists a fully symmetric linear functional $\varphi:E\to\mathbb{R}$ such that
$$\varphi(x)=\lim_{m\to\infty}\frac1m\|\sigma_m(\mu(x))\|_E.$$
\end{thm}
\begin{proof} Without loss of generality, $x=\mu(x).$ Let $q$ be the convex monotone functional constructed in Lemma \ref{required q}. It follows from Lemma \ref{positive hahn-banach} that there exist a positive linear functional $\varphi$ on $E$ such that $\varphi\leq q$ and $\varphi(x)=q(x).$

It is clear that $\varphi\leq q\leq p.$ Since $p(z)=0$ for every $z\in Z_E,$ it follows that $\varphi(z)=0$ for every $z\in Z_E.$ Therefore, $\varphi$ is a symmetric functional. For every $z\in\mathcal{D}_E$ with $Cz\leq0,$ we have $\varphi(z)\leq q(z)\leq p(0)=0.$

Let $x_1,x_2\in E$ be positive elements such that $x_1\prec\prec x_2.$ Therefore, $z=\mu(x_1)-\mu(x_2)\in\mathcal{D}_E$ and $Cz\leq0.$ It follows from above that $\varphi(z)\leq 0.$ Hence, $\varphi$ is a fully symmetric functional.

Since $\varphi(z)\leq q(z)\leq p(z)\leq\|z\|_E$ for every $z=\mu(z)\in E,$ it follows that $\|\varphi\|_{E^*}\leq 1.$ Therefore,
$$\varphi(x)=\varphi(\frac1m\sigma_mx)\leq\frac1m\|\sigma_mx\|_E.$$
Passing $m\to\infty,$ we obtain
$$\varphi(x)\leq\lim_{m\to\infty}\frac1m\|\sigma_m\mu(x)\|_E.$$
On the other hand, $q(x)=p(x)$ by Lemma \ref{ph monotone}. By Lemma \ref{mn estimate}, we have $m^{-1}\sigma_mx\lhd M_mx.$ Therefore,
$$\varphi(x)=q(x)=p(x)=\limsup_{m\to\infty}\|M_mx\|_E\geq\lim_{m\to\infty}\frac1m\|\sigma_m\mu(x)\|_E.$$
The assertion follows immediately. 
\end{proof}

If $\pi:E\to E$ is a convex functional defined in \eqref{pi deff}, then $\pi(-x)=\pi(x)$ for every $x\in E.$ If $q$ is a functional defined in Lemma \ref{required q}, then $q(-x)=0$ for positive $x\in E.$ Therefore, $q\neq\pi.$ However, the assertion below follows from Theorem \ref{fs exists}.

\begin{lem}\label{pietsch request2} Let $E=E(0,\infty)$ be a fully symmetric Banach space on the semi-axis. Let $q$ and $\pi$ be the convex functionals on $E$ defined in Lemma \ref{required q} and \eqref{pi deff}, respectively. For every positive $x\in E,$ we have $q(x)=\pi(x).$
\end{lem}
\begin{proof} For every $x\in E,$ consider the functional $\varphi$ constructed in Theorem \ref{fs exists}. By construction, we have $\varphi(x)=q(x)=p(x)=\pi(x).$
\end{proof}

The proofs of the two following theorems are very similar to that of Theorem \ref{sim sing exists inf} (respectively, Theorem \ref{sim sing exists fin}) and are, therefore, omitted. The only difference is that the reference to Theorem \ref{sim exists} (respectively, Theorem \ref{sim sing exists inf}) has to be replaced with the reference to Theorem \ref{fs exists} (respectively, Theorem \ref{fs sing exists inf}).

\begin{thm}\label{fs sing exists inf} Let $E\subset L_1(0,\infty)$ be a fully symmetric Banach space on the semi-axis. For a given $0\leq x\in E,$ there exists a singular fully symmetric linear functional $\varphi_{sing}$ such that
$$\varphi_{sing}(x)=\lim_{m\to\infty}\frac1m\|\sigma_m(\mu(x))\chi_{(0,1)}\|_E.$$
\end{thm}

\begin{thm}\label{fs sing exists fin} Let $E\subset L_1(0,1)$ be a fully symmetric Banach space on the interval $(0,1).$ For a given $0\leq x\in E,$ there exists a singular fully symmetric linear functional $\varphi_{sing}$ such that
$$\varphi_{sing}(x)=\lim_{m\to\infty}\frac1m\|\sigma_m(\mu(x))\|_E.$$
\end{thm}

\section{The sets of symmetric and fully symmetric functionals are different}\label{ne vpol sim}

In this section, we demonstrate that the sets of symmetric and fully symmetric functionals on a given fully symmetric space $E$ are distinct (provided that one of these sets is non-empty). The main results are Theorem \ref{nonintegrable theorem} and Theorem \ref{integrable theorem}.

Let $x=\mu(x)\in (L_1+L_{\infty})(0,\infty)$ (or $x=\mu(x)\in L_1(0,1)$) and let $X(t)=\int_0^tx(s)ds.$ For every $\theta>0,$ let $a_n(\theta)$ be such that
$$X(a_n(\theta))=(3/2)^n\theta$$
for every $n\in\mathbb{Z}$ such that $a_n(\theta)$ does exist. Given a sequence $\kappa=\{\kappa_n\}_{n\in\mathbb{Z}}\in(\mathbb{N}\cup\{\infty\})^{\mathbb{Z}},$ let
$$\mathcal{B}_{\kappa,\theta}=\{\kappa_na_{3n}(\theta),\mbox{ where }n\in\mathbb{Z}\mbox{ is such that } \kappa_n^2a_{3n}(\theta)<a_{3n+1}(\theta)\}.$$
If $\kappa_n=m$ for all $n\in\mathbb{N},$ we write $\mathcal{B}_{m,\theta}$ instead of $\mathcal{B}_{\kappa,\theta}.$ Also, set
$$\mathcal{A}_m=\{ma_n(1):\quad m^2a_n(1)<a_{n+1}(1),\ n\in\mathbb{Z}\}.$$

\begin{lem}\label{union lemma} If $x=\mu(x)\in L_1+L_{\infty}$ and if $\mathcal{C}_i,$ $1\leq i\leq k,$ are discrete sets, then
$$\mathbf{E}(x|\cup_{i=1}^k\mathcal{C}_i)\prec\prec\sum_{i=1}^k\mathbf{E}(x|\mathcal{C}_i).$$
\end{lem}
\begin{proof} It is sufficient to verify
$$\int_0^t\mathbf{E}(x|\cup_{i=1}^k\mathcal{C}_i)(s)ds\leq\sum_{i=1}^k\int_0^t\mathbf{E}(x|\mathcal{C}_i)(s)ds$$
only at the nodes of $\mathbf{E}(x|\cup_{i=1}^k\mathcal{C}_i),$ that is at the nodes of
$\mathbf{E}(x|\mathcal{C}_i)$ for every $i.$ However, if $t\in\mathcal{C}_i$ for some $i,$ then
$$\int_0^t\mathbf{E}(x|\cup_{i=1}^k\mathcal{C}_i)(s)ds=X(t)=\int_0^t\mathbf{E}(x|\mathcal{C}_i)(s)ds$$
and we are done. 
\end{proof}

We will need the following lemma.

\begin{lem}\label{majorant lemma} If $x=\mu(x)\in L_1+L_{\infty}$ and if $\kappa\geq\kappa'$ (that is $\kappa_n\geq\kappa_n'$ for every $n$), then
\begin{equation}\label{monotone}
\mathbf{E}(x|\mathcal{B}_{\kappa,\theta})\prec\prec\frac32\mathbf{E}(x|\mathcal{B}_{\kappa',\theta}).
\end{equation}
\end{lem}

\begin{proof} Let $n\in\mathbb{Z}$ be such that $\kappa_n^2a_{3n}(\theta)<a_{3n+1}(\theta).$ It follows that $\kappa_n'^2a_{3n}(\theta)<a_{3n+1}(\theta).$ Therefore,
$$\int_0^{\kappa_na_{3n}(\theta)}\mathbf{E}(x|\mathcal{B}_{\kappa,\theta})(s)ds\leq\int_0^{a_{3n+1}(\theta)}x(s)ds=3/2\int_0^{a_{3n}(\theta)}x(s)ds\leq$$
$$\leq3/2\int_0^{\kappa_n'a_{3n}(\theta)}x(s)ds=3/2\int_0^{\kappa_n'a_{3n}(\theta)}\mathbf{E}(x|\mathcal{B}_{\kappa',\theta})(s)ds.$$
Hence, we have
\begin{equation}\label{qss11}
\int_0^t\mathbf{E}(x|\mathcal{B}_{\kappa,\theta})(s)ds\leq3/2\int_0^t\mathbf{E}(x|\mathcal{B}_{\kappa',\theta})(s)ds
\end{equation}
for every $t$ being a node of the partition $\mathcal{B}_{\kappa,\theta}.$ Thus, \eqref{qss11} holds for every $t>0.$ 
\end{proof}

\begin{rem}\label{majorant remark} The inequality \eqref{monotone} holds if $\kappa_n\geq\kappa_n'$ only for such $n\in\mathbb{Z}$ that satisfy the inequality $\kappa_n^2a_{3n}(\theta)<a_{3n+1}(\theta).$
\end{rem}

\begin{lem}\label{characterization lemma} Let $E$ be a fully symmetric Banach function space either on the interval $(0,1)$ or on the semi-axis. Let $x=\mu(x)\in E$ and $y=\mu(y)\in E$ be such that $\varphi(y)\leq\varphi(x)$ for every positive symmetric functional $\varphi\in E^*.$ There exists $0\leq u_m\in E$ such that $u_m\to 0$ in $E$ and
$$\int_{ma}^by(s)ds\leq\int_a^{mb}(x+u_m)(s)ds,\quad \forall ma\leq b.$$
\end{lem}
\begin{proof} Let $p$ be a convex positive functional considered in Lemma \ref{required p}. By Lemma \ref{positive hahn-banach}, there exists a positive functional $\varphi\in E^*$ such that $\varphi\leq p$ and $\varphi(y-x)=p(y-x).$ We have $p(z)=0$ for every $z\in Z_E$ and, therefore, $\varphi(z)=0$ for every $z\in Z_E.$ Therefore, $\varphi$ is a positive symmetric linear functional on $E.$

By the assumption, $\varphi(y-x)\leq0$ and, therefore, $p(y-x)=0.$ Hence, by the definition of $p,$ we have $u_m=(M_m(y-x))_+\to0$ in $E.$ Clearly, $M_my\leq M_mx+u_m.$ It follows from Lemma \ref{mn estimate} that
$$\int_{ma}^by(s)ds\leq\int_{ma}^{mb}(M_my)(s)ds\leq\int_{ma}^{mb}(M_mx+u_m)(s)ds\leq\int_{a}^{mb}(x+u_m)(s)ds.$$

\end{proof}

For each sequence $\kappa$ and $\lambda>0,$ we define the sequence $\kappa^{\lambda}$ by setting
$$\kappa^{\lambda}_n=
\begin{cases}
\kappa_n,\qquad \kappa_n\ge\lambda\\
\infty,\qquad \kappa_n<\lambda.
\end{cases}$$

\begin{lem}\label{main technical estimate} If $m\in\mathbb{N},$ $x=\mu(x)\in L_1+L_{\infty}$ and $0\leq u\in L_1+L_{\infty}$ are such that
$$\int_{ma}^b\mathbf{E}(x|\mathcal{B}_{\kappa,\theta})(s)ds\leq\int_a^{mb}(x+u)(s)ds,\quad \forall ma\leq b\in\mathbb{R},$$
then
\begin{equation}\label{kappa estimate}
m^{-1}\sigma_m\mathbf{E}(x|\mathcal{B}_{\kappa^{100m},\theta})\prec\prec 30\mu(u).
\end{equation}
\end{lem}
\begin{proof} If $\kappa^{100m}_n=\infty$ for every $n\in\mathbb{Z},$ then $\mathbf{E}(x|\mathcal{B}_{\kappa^{100m},\theta})=0$ and the assertion is trivial.

Let $n\in\mathbb{Z}$ be such that $\kappa_n^2a_{3n}(\theta)<a_{3n+1}(\theta)$ and $\kappa_n\geq 100m.$ It follows that
\begin{equation}\label{qss12}
\int_0^{m\kappa_na_{3n}(\theta)}u(s)ds\geq\int_{a_{3n}(\theta)}^{m\kappa_na_{3n}(\theta)}(x+u)(s)ds-\int_{a_{3n}(\theta)}^{m\kappa_na_{3n}(\theta)}x(s)ds.
\end{equation}
By the assumption, we have
\begin{equation}\label{qss13}
\int_{a_{3n}(\theta)}^{m\kappa_na_{3n}(\theta)}(x+u)(s)ds\geq\int_{ma_{3n}(\theta)}^{\kappa_na_{3n}(\theta)}\mathbf{E}(x|\mathcal{B}_{\kappa,\theta})(s)ds.
\end{equation}
Note that $m\kappa_na_{3n}(\theta)<a_{3n+1}(\theta).$ It follows from \eqref{qss12} and \eqref{qss13} that
\begin{equation}\label{qss14}
\int_0^{m\kappa_na_{3n}(\theta)}u(s)ds\geq\int_{ma_{3n}(\theta)}^{\kappa_na_{3n}(\theta)}\mathbf{E}(x|\mathcal{B}_{\kappa,\theta})(s)ds-\int_{a_{3n}(\theta)}^{a_{3n+1}(\theta)}x(s)ds.
\end{equation}
Let $n'$ be the maximal integer number such that $n'<n$ and $\kappa_{n'}^2a_{3n'}(\theta)<a_{3n'+1}(\theta).$ It is clear that
$$\kappa_{n'}^2a_{3n'}(\theta)<a_{3n'+1}(\theta)\leq a_{3n-2}(\theta)<ma_{3n}(\theta)$$
and
\begin{equation}\label{qss15} \mathbf{E}(x|\mathcal{B}_{\kappa,\theta})=\frac{X(\kappa_na_{3n}(\theta))-X(\kappa_{n'}a_{3n'}(\theta))}{\kappa_na_{3n}(\theta)-\kappa_{n'}a_{3n'}(\theta)}\geq\frac{X(a_{3n}(\theta))-X(a_{3n-2}(\theta))}{\kappa_na_{3n}(\theta)}
\end{equation}
on the interval $(ma_{3n}(\theta),\kappa_na_{3n}(\theta)).$

\noindent If $\kappa_{n'}^2a_{3n'}(\theta)\geq a_{3n+1}(\theta)$ for every $n'<n,$ then
\begin{equation}\label{qss151}
\mathbf{E}(x|\mathcal{B}_{\kappa,\theta})=\frac{X(\kappa_na_{3n}(\theta))}{\kappa_na_{3n}(\theta)}\geq\frac{X(a_{3n}(\theta))}{\kappa_na_{3n}(\theta)}
\end{equation}
on the interval $(ma_{3n}(\theta),\kappa_na_{3n}(\theta)).$

It follows from \eqref{qss14} and \eqref{qss15} (or \eqref{qss151}) that
$$\int_0^{m\kappa_na_{3n}(\theta)}u(s)ds\geq\frac{\kappa_n-m}{\kappa_n}\cdot(1-\frac49)X(a_{3n}(\theta))-\frac12 X(a_{3n}(\theta)).$$
Since $\kappa_n\geq 100m,$ it follows that
$$\int_0^{m\kappa_na_{3n}(\theta)}u(s)ds\geq((1-\frac1{100})(1-\frac49)-\frac12)X(a_{3n}(\theta))=\frac1{20}X(a_{3n}(\theta))=$$
$$=\frac1{30}X(a_{3n+1}(\theta))\geq\frac1{30}X(\kappa_na_{3n}(\theta))=\frac1{30}\int_0^{\kappa_na_{3n}(\theta)}\mathbf{E}(x|\mathcal{B}_{\kappa^{100m},\theta})(s)ds.$$
It follows immediately that
\begin{equation}\label{qss16}
\int_0^t\mathbf{E}(x|\mathcal{B}_{\kappa^{100m},\theta})(s)ds\leq 30\int_0^{mt}u(s)ds\leq 30\int_0^{mt}\mu(s,u)ds
\end{equation}
for every $t$ being a node of the partition $\mathcal{B}_{\kappa^{100m},\theta}.$ Therefore,
$$\int_0^t\mathbf{E}(x|\mathcal{B}_{\kappa^{100m},\theta})(s)ds\leq30\int_0^{mt}\mu(s,u)ds,\quad t>0$$
or, equivalently,
$$\int_0^{t/m}\mathbf{E}(x|\mathcal{B}_{\kappa^{100m},\theta})(s)ds\leq30\int_0^t\mu(s,u)ds,\quad t>0.$$
The assertion follows immediately. 
\end{proof}

\begin{lem}\label{kn convergence} Let $E$ be a fully symmetric Banach function space either on the interval $(0,1)$ or on the semi-axis. If $x=\mu(x)\in E$ is such that $\varphi(y)\leq\varphi(x)$ for every positive symmetric functional $\varphi$ on $E$ and every $0\leq y\prec\prec x,$ then $\lambda^{-1}\sigma_{\lambda}\mathbf{E}(x|\mathcal{B}_{\kappa^{\lambda},\theta})\to0$ as $\lambda\to\infty.$
\end{lem}
\begin{proof} Since $\mathbf{E}(x|\mathcal{B}_{\kappa,\theta})\prec\prec x,$ it follows from the assumption and Lemma \ref{characterization lemma} that there exists $0\leq u_m\to0$ such that
$$\int_{ma}^b\mathbf{E}(x|\mathcal{B}_{\kappa,\theta})(s)ds\leq\int_a^{mb}(x+u_m)(s)ds,\quad\forall ma\leq b\in\mathbb{R}.$$
For every $\lambda\geq 100m,$ we have $\kappa^{100m}\leq\kappa^{\lambda}.$ It follows from Lemma \ref{main technical estimate} that
$$\frac1{\lambda}\sigma_{\lambda}\mathbf{E}(x|\mathcal{B}_{\kappa^{\lambda},\theta})\prec\prec\frac1m\sigma_m\mathbf{E}(x|\mathcal{B}_{\kappa^{\lambda},\theta})\stackrel{Lemma \ref{majorant lemma}}{\prec\prec}\frac3{2m}\sigma_m\mathbf{E}(x|\mathcal{B}_{\kappa^{100m},\theta})\stackrel{\eqref{kappa estimate}}{\prec\prec}45\mu(u_m).$$
The assertion now follows immediately. 
\end{proof}

\begin{prop}\label{pbm convergence} Let $E$ be a fully symmetric Banach function space either on the interval $(0,1)$ or on the semi-axis equipped with a Fatou norm. If $x=\mu(x)\in E$ is such that $\varphi(y)\leq\varphi(x)$ for every positive symmetric functional $\varphi$ on $E$ and every $0\leq y\prec\prec x,$ then $m^{-1}\sigma_m\mathbf{E}(x|\mathcal{B}_{m,\theta})\to0$ as $m\to\infty.$
\end{prop}

\begin{proof} For every $m,r\in\mathbb{N},$ set
$$
\kappa^{m,r}_n=
\begin{cases}
m \qquad 0\le |n|<r\\
\infty \qquad r \leq|n|
\end{cases}
$$
and $\kappa^{m,r}=\{\kappa^{m,r}_n\}_{n\in\mathbb{Z}}.$ Clearly, $\mathbf{E}(x|\mathcal{B}_{\kappa^{m,r},\theta})\to \mathbf{E}(x|\mathcal{B}_{m,\theta})$ almost everywhere when $r\to\infty.$ It follows from the definition of Fatou norm that
$$\lim_{r\to\infty}\|\sigma_m\mathbf{E}(x|\mathcal{B}_{\kappa^{m,r},\theta})\|_E=\|\sigma_m\mathbf{E}(x|\mathcal{B}_{m,\theta})\|_E.$$
Select $r_m$ so large that
\begin{equation}\label{qss17}
\frac1m\|\sigma_m\mathbf{E}(x|\mathcal{B}_{\kappa^{m,r_m},\theta})\|_E>\frac1{2m}\|\sigma_m\mathbf{E}(x|\mathcal{B}_{m,\theta})\|_E.
\end{equation}
Now define the sequence $\kappa=\{\kappa_n\}_{n\in\mathbb{Z}}$ by setting
$$\kappa_n=\inf_{m\geq1}\kappa^{m,r_m}_n=\inf_{r_m>|n|}m,\quad n\in\mathbb{Z}.$$
Clearly, $r_{\kappa_n}\geq|n|$ and, therefore, $\kappa_n\rightarrow\infty$ as $|n|\rightarrow\infty.$ In particular, the set $\{n:\kappa_n<\lambda\}$ is finite for every $\lambda\in\mathbb{N}.$ Set
$$M(\lambda)=\max\{\lambda,\max_{\kappa_n<\lambda}(\frac{a_{3n+1}(\theta)}{a_{3n}(\theta)})^{1/2}\}.$$
If $m>M(\lambda),$ then $m^2a_{3n}(\theta)\geq a_{3n+1}(\theta)$ whenever $\kappa_n<\lambda.$ Thus, $\kappa_n\geq\lambda$ whenever $m^2a_{3n}(\theta)<a_{3n+1}(\theta).$ Hence, $\kappa_n^{\lambda}=\kappa_n$ whenever $(\kappa_n^{m,r_m})^2a_{3n}(\theta)<a_{3n+1}(\theta).$ Therefore, $\kappa_n^{\lambda}\leq\kappa_n^{m,r_m}$ for every $n\in\mathbb{Z}$ such that $(\kappa_n^{m,r_m})^2a_{3n}(\theta)<a_{3n+1}(\theta).$ According to Remark \ref{majorant remark}, it follows that
$$\mathbf{E}(x|\mathcal{B}_{\kappa^{m,r_m},\theta})\prec\prec\frac32\mathbf{E}(x|\mathcal{B}_{\kappa^{\lambda},\theta}).$$
Since $m\geq\lambda,$ it follows that
\begin{equation}\label{qss18}
\frac1m\sigma_m\mathbf{E}(x|\mathcal{B}_{\kappa^{m,r_m},\theta})\prec\prec\frac3{2\lambda}\sigma_{\lambda}\mathbf{E}(x|\mathcal{B}_{\kappa^{\lambda},\theta}).
\end{equation}
By Lemma \ref{kn convergence}, for every $\varepsilon>0,$ there exists $\lambda$ such that
\begin{equation}\label{qss19}
\frac1{\lambda}\|\sigma_{\lambda}\mathbf{E}(x|\mathcal{B}_{\kappa^{\lambda},\theta})\|_E<\frac13\varepsilon.
\end{equation}
It follows that
$$\frac1m\|\sigma_m\mathbf{E}(x|\mathcal{B}_{m,\theta})\|_E\stackrel{\eqref{qss17}}{\leq}\frac2m\|\sigma_m\mathbf{E}(x|\mathcal{B}_{\kappa^{m,r_m},\theta})\|_E\stackrel{\eqref{qss18}}{\leq}\frac3{\lambda}\|\sigma_{\lambda}\mathbf{E}(x|\mathcal{B}_{\kappa^{\lambda},\theta})\|_E\stackrel{\eqref{qss19}}{<}\varepsilon$$
for every $m>M(\lambda).$ Since $\varepsilon>0$ is arbitrarily small, the assertion follows. 
\end{proof}

\begin{lem}\label{pam convergence} Let $E$ be a fully symmetric Banach space either on the interval $(0,1)$ or on the semi-axis equipped with a Fatou norm. If $x=\mu(x)\in E$ is such that $\varphi(y)\leq\varphi(x)$ for every positive symmetric functional $\varphi$ on $E$ and every $0\leq y\prec\prec x,$ then $m^{-1}\sigma_m\mathbf{E}(x|\mathcal{A}_m)\to0$ as $m\to\infty.$
\end{lem}
\begin{proof} It is clear that $a_k(3/2)=a_{k+1}(1)$ and $a_k((3/2)^2)=a_{k+2}(1)$ for every $k\in\mathbb{N}.$ It follows that
$$\mathcal{B}_{m,1}\cup\mathcal{B}_{m,3/2}\cup\mathcal{B}_{m,(3/2)^2}=\mathcal{A}_m.$$
Therefore, by Lemma \ref{union lemma}, we have
\begin{equation}\label{pam estimate}
\mathbf{E}(x|\mathcal{A}_m)\prec\prec \mathbf{E}(x|\mathcal{B}_{m,1})+\mathbf{E}(x|\mathcal{B}_{m,3/2})+\mathbf{E}(x|\mathcal{B}_{m,(3/2)^2}).
\end{equation}
The assertion follows now from Proposition \ref{pbm convergence}. 
\end{proof}

\begin{lem}\label{prec estimate} Let $x=\mu(x)\in L_1+L_{\infty}(0,\infty)$ be a function on the semi-axis. If $x\notin L_1(0,\infty),$ then, for every $t>0$ and every $m\in\mathbb{N},$ we have
\begin{equation}
X(t)\leq\frac23X(m^4t)+\frac32\int_0^{m^4t}\mathbf{E}(x|\mathcal{A}_m)(s)ds.
\end{equation}
\end{lem}
\begin{proof} For a given $t>0,$ there exists $n\in\mathbb{Z}$ such that $t\in[a_n(1),a_{n+1}(1)].$ If $a_{n+1}(1)> m^2a_n(1),$ then
$$\int_0^{m^4t}\mathbf{E}(x|\mathcal{A}_m)(s)ds\geq\int_0^{ma_n(1)}\mathbf{E}(x|\mathcal{A}_m)(s)ds=X(ma_n(1))\geq\frac23X(t).$$
If $a_{n+1}(1)\leq m^2a_n(1)$ and $a_{n+2}(1)>m^2a_{n+1}(1),$ then
$$\int_0^{m^4t}\mathbf{E}(x|\mathcal{A}_m)(s)ds\geq\int_0^{ma_{n+1}(1)}\mathbf{E}(x|\mathcal{A}_m)(s)ds=X(ma_{n+1}(1))\geq X(t).$$
If $a_{n+2}(1)\leq m^2a_{n+1}(1)$ and $a_{n+1}(1)\leq m^2a_n(1),$ then
$$X(m^4t)\geq X(a_{n+2}(1))=\frac32 X(a_{n+1}(1))\geq\frac32 X(t)$$
and the assertion follows. 
\end{proof}

The situation in the case that $x\in L_1$ is slightly more complicated.

\begin{lem}\label{prec estimate integrable} If $x=\mu(x)\in L_1(0,1)$ or $x\in L_1(0,\infty),$ then there exists constant $C$ such that for every $t>0$
\begin{equation}
X(t)\leq\frac23X(m^4t)+\frac32\int_0^{m^4t}\mathbf{E}(x|\mathcal{A}_m)(s)ds+C\int_0^{m^4t}\chi_{[0,1]}(s)ds.
\end{equation}
\end{lem}
\begin{proof} Consider first the case of the semi-axis. Fix $n_0$ such that $X(a_{n_0})\leq 4/9 X(\infty).$ For a givne $t\in[a,a_{n_0}],$ there exists $n\in\mathbb{Z}$ such that $n<n_0$ and $t\in[a_n,a_{n+1}].$ Then, the argument in Lemma \ref{prec estimate} applies {\it mutatis mutandi}. For every $t\geq a_{n_0}$ we have
$$X(t)\leq\frac{X(\infty)}{\min\{a_{n_0},1\}}\min\{m^4t,1\}=\frac{X(\infty)}{\min\{a_{n_0},1\}}\int_0^{m^4t}\chi_{[0,1]}(s)ds.$$
Setting $C=X(\infty)/\min\{a_{n_0},1\},$ we obtain the assertion.

The same argument applies in the case of the interval $(0,1)$ by replacing $X(\infty)$ by $X(1).$ 
\end{proof}

The following two theorems are crucial for the proof of the implication $(3)\Leftrightarrow(4)$ in Theorem \ref{main theorem of the paper}.

\begin{thm}\label{nonintegrable theorem} Let $E$ be a fully symmetric Banach space either on the interval $(0,1)$ or on the semi-axis  and let $x\in E.$ Suppose that the norm on $E$ is a Fatou norm. If $\varphi(y)\leq\varphi(x)$ for every positive symmetric functional on $E$ and every $0\leq y\prec\prec x,$ then
\begin{equation}\label{main condition from orbits}
\lim_{m\to\infty}\frac1m\|\sigma_m(\mu(x))\|_E=0
\end{equation}
provided that one of the following conditions is satisfied
\begin{enumerate}
\item $E=E(0,1)$ is a space on the interval $(0,1).$
\item $E=E(0,\infty)$ is a space on the semi-axis and $E(0,\infty)\not\subset L_1(0,\infty).$
\end{enumerate}
\end{thm}
\begin{proof} Without loss of generality, $x=\mu(x).$ If $x\notin L_1,$ then by Lemma \ref{prec estimate},
$$\int_0^{t/m^4}x(s)ds\leq\frac23\int_0^tx(s)ds+\frac32\int_0^t\mathbf{E}(x|\mathcal{A}_m)(s)ds,\quad\forall t>0$$
or, equivalently,
$$\frac1{m^4}\sigma_{m^4}x\prec\prec\frac23x+\frac32\mathbf{E}(x|\mathcal{A}_m).$$
Applying $m^{-1}\sigma_m$ to the both parts, we obtain
$$\frac1{m^5}\sigma_{m^5}x\prec\prec\frac23\frac1m\sigma_mx+\frac32 \frac1m\sigma_m\mathbf{E}(x|\mathcal{A}_m).$$
Take norms and let $m\to\infty.$ It follows from Lemma \ref{pam convergence} that
$$\lim_{m\to\infty}\frac1m\|\sigma_mx\|_E\leq\frac23\lim_{m\to\infty}\frac1m\|\sigma_mx\|_E.$$
This proves \eqref{main condition from orbits}.

If $x\in L_1$ and $C$ are as in Lemma \ref{prec estimate integrable}, then it follows from Lemma \ref{prec estimate integrable} that
$$\int_0^{t/m^4}x(s)ds\leq\frac23\int_0^tx(s)ds+\frac32\int_0^t\mathbf{E}(x|\mathcal{A}_m)(s)ds+C\int_0^t\chi_{[0,1]}(s)ds,\quad\forall t>0$$
or, equivalently,
$$\frac1{m^4}\sigma_{m^4}x\prec\prec\frac23x+\frac32\mathbf{E}(x|\mathcal{A}_m)+C\chi_{(0,1)}.$$
Applying $m^{-1}\sigma_m$ to the both parts, we obtain
$$\frac1{m^5}\sigma_{m^5}x\prec\prec\frac23\frac1m\sigma_mx+\frac32\frac1m\sigma_m\mathbf{E}(x|\mathcal{A}_m)+C\frac1m\sigma_m\chi_{(0,1)}.$$
Take norms and let $m\to\infty.$ For every symmetric space $E$ on the interval $(0,1)$ and for every symmetric space $E$ on the semi-axis such that $E\not\subset L_1(0,\infty)$ we have $m^{-1}\|\sigma_m\chi_{(0,1)}\|_E\to0.$ It follows from Lemma \ref{pam convergence} that
$$\lim_{m\to\infty}\frac1m\|\sigma_mx\|_E\leq\frac23\lim_{m\to\infty}\frac1m\|\sigma_mx\|_E$$
and again \eqref{main condition from orbits} follows. 
\end{proof}

\begin{thm}\label{integrable theorem} Let $E=E(0,\infty)$ be a fully symmetric Banach space on the semi-axis equipped with a Fatou norm such that $E(0,\infty)\subset L_1(0,\infty).$ If $\varphi(y)\leq\varphi(x)$ for every positive symmetric functional on $E$ and every $0\leq y\prec\prec x,$ then
\begin{equation}\label{main condition from orbits integrable}
\lim_{m\to\infty}\frac1m\|\sigma_m(\mu(x))\chi_{(0,1)}\|_E=0.
\end{equation}
\end{thm}
\begin{proof} Fully symmetric Banach space $F$ on the interval $(0,1)$ consists of those $z\in E$ supported on the interval $(0,1).$  Let $x_1=\mu(x)\chi_{(0,1)}\in F.$ Suppose that
$$\lim_{m\to\infty}\frac1m\|\sigma_m(\mu(x))\chi_{(0,1)}\|_E>0.$$
It clearly follows that
$$\lim_{m\to\infty}\frac1m\|\sigma_m(\mu(x_1))\|_F>0.$$
By Theorem \ref{nonintegrable theorem}, there exists $0\leq y_1\prec\prec x_1$ and a positive symmetric functional $\varphi\in F^*$ such that $\varphi(y_1)>\varphi(x_1).$ Let $\varphi_{sing}$ be a singular part of the functional $\varphi$ constructed in Lemma \ref{sing constr}. It follows from Lemma \ref{sing constr} that $\varphi_{sing}$ is symmetric. By Lemma \ref{normal construction}, the difference $\varphi-\varphi_{sing}$ is a symmetric normal functional on $F$ (that is, an integral). Therefore, $\varphi_{sing}(y_1)>\varphi_{sing}(x_1).$

Now we show that the functional $\varphi_{sing}$ can be extended from $F$ to $E$ by setting
$$\varphi_{sing}(z)=\lim_{n\to\infty}\varphi_{sing}(\mu(z)\chi_{(0,1/n)}),\quad 0\leq z\in E.$$
Repeating the argument in Lemma \ref{sing constr}, we prove that the extension above is additive on $E_+.$ Thus, the functional $\varphi_{sing}\in E^*$ is positive and symmetric. Since $y_1\prec\prec x$ and $\varphi_{sing}(y_1)>\varphi_{sing}(x_1)=\varphi_{sing}(x),$ the assertion follows. 
\end{proof}

\section{Proof of Theorem \ref{main theorem of the paper}}\label{prmain}

In this section, we prove an assertion more general then that of Theorem \ref{main theorem of the paper}. The assertion of Theorem \ref{main theorem of the paper} follows from that of Theorem \ref{obshiy variant} by setting $\mathcal{M}=B(H).$

In what follows, the semifinite von Neumann algebra $\mathcal{M}$ is either atomless or atomic so that the trace of every atom is $1.$

\begin{thm}\label{obshiy variant} Let $E(\mathcal{M},\tau)$ be a symmetric operator space. Consider the following conditions.
\begin{enumerate}
\item\label{first main condition} There exist nontrivial positive singular symmetric functionals on $E(\mathcal{M},\tau).$
\item\label{second main condition} There exist nontrivial singular fully symmetric functionals on $E(\mathcal{M},\tau).$
\item\label{third main condition} There exist positive symmetric symmetric functional on $E(\mathcal{M},\tau)$ which are not fully symmetric.
\item\label{fourth main condition} If $E(\mathcal{M},\tau)\not\subset L_1(\mathcal{M},\tau),$ then there exists an operator $A\in E(\mathcal{M},\tau)$ such that
\begin{equation}\label{nonint sigma}
\lim_{m\to\infty}\frac1m\|\sigma_m\mu(A)\|_E>0.
\end{equation}
If $E(\mathcal{M},\tau)\subset L_1(\mathcal{M},\tau),$  then there exists an operator $A\in E(\mathcal{M},\tau)$ such that
\begin{equation}\label{int sigma}
\lim_{m\to\infty}\frac1m\|(\sigma_m\mu(A))\chi_{(0,1)}\|_E>0.
\end{equation}
\end{enumerate}
\begin{enumerate}[(i)]
\item  The conditions \eqref{first main condition} and \eqref{fourth main condition} are equivalent for every symmetric operator space $E(\mathcal{M},\tau).$
\item  The conditions \eqref{first main condition}, \eqref{second main condition} and \eqref{fourth main condition} are equivalent for every fully symmetric operator space $E(\mathcal{M},\tau).$
\item  The conditions \eqref{first main condition}-\eqref{fourth main condition} are equivalent for every fully symmetric operator space $E(\mathcal{M},\tau)$ equipped with a Fatou norm.
\end{enumerate}
\end{thm}
\begin{proof} Implications $\eqref{second main condition}\Rightarrow\eqref{first main condition}$ and $\eqref{third main condition}\Rightarrow\eqref{first main condition}$ are trivial.

$\eqref{first main condition}\Rightarrow\eqref{fourth main condition}$ Let $E(\mathcal{M},\tau)$ be a symmetric operator space with a singular symmetric functional $\varphi.$ Let $A\in E(\mathcal{M},\tau)$ be an operator such that $\varphi(A)\neq0.$ Without loss of generality, $A\geq 0.$

If $E(\mathcal{M},\tau)\not\subset L_1(\mathcal{M},\tau),$ then
$$|\varphi(A)|=\frac1m|\varphi(\underbrace{A\oplus\cdots\oplus A}_{\mbox{$m$ times}})|\leq\|\varphi\|_{E^*(\mathcal{M},\tau)}\cdot\frac1m\|\sigma_m\mu(A)\|_E.$$
Passing $m\to\infty,$ we obtain the required inequality \eqref{nonint sigma}.

Let now $E(\mathcal{M},\tau)\subset L_1(\mathcal{M},\tau).$ If $\mathcal{M}$ is atomic, then $E(\mathcal{M},\tau)= L_1(\mathcal{M},\tau)$ and the assertion is trivial. Let $\mathcal{M}$ be atomless. Since $\varphi$ is a singular functional and $$A-AE_A(\mu(\frac1m,A),\infty)\in(L_1\cap L_{\infty})(\mathcal{M},\tau),\quad\forall m\in\mathbb{N},$$
we infer that
$$|\varphi(A)|=|\varphi(AE_A(\mu(\frac1m,A),\infty))|=$$
$$=\frac1m|\varphi(\underbrace{AE_A(\mu(\frac1m,A),\infty)\oplus\cdots\oplus AE_A(\mu(\frac1m,A),\infty)}_{\mbox{$m$ times}})|\leq$$
$$\leq\|\varphi\|_{E^*(\mathcal{M},\tau)}\cdot\frac1m\|\sigma_m\mu(AE_A(\mu(\frac1m,A),\infty))\|_E\leq\|\varphi\|_{E^*(\mathcal{M},\tau)}\cdot\frac1m\|(\sigma_m\mu(A))\chi_{(0,1)}\|_E.$$
Passing $m\to\infty,$ we obtain the required inequality \eqref{int sigma}.

$\eqref{fourth main condition}\Rightarrow\eqref{first main condition}$ Firstly, we assume that the algebra $\mathcal{M}$ is finite. Without loss of generality, $\tau(1)=1.$ Let $E(\mathcal{M},\tau)$ be a symmetric operator space and let $E(0,1)$ be the corresponding symmetric function space. By the assumption, there exists an element $x=\mu(A)\in E(0,1)$ such that $m^{-1}\sigma_mx\not\to0$ in $E(0,1).$ By Theorem \ref{sim sing exists fin}, there exists a positive singular symmetric functional $0\neq\varphi\in E(0,1)^*.$ Let $\mathcal{L}(\varphi)$ be a functional on $E(\mathcal{M},\tau)$ defined in Theorem \ref{lifting theorem}. Clearly, $\mathcal{L}(\varphi)$ is a nontrivial positive symmetric functional on $E(\mathcal{M},\tau).$

The case when $\mathcal{M}$ is an infinite atomless von Neumann algebra can be treated in a similar manner. The only difference is that the reference to Theorem \ref{sim sing exists fin} has to be replaced with the reference to either Theorem \ref{sim sing exists inf} or Theorem \ref{sim exists}.

Let $E(\mathcal{M},\tau)$ be a symmetric operator space on a atomic von Neumann algebra $\mathcal{M}$ and let $E(\mathbb{N})$ be the corresponding symmetric sequence space. It follows from the assumption that $E(\mathcal{M},\tau)\neq L_1(\mathcal{M},\tau)$ or, equivalently, $E(\mathbb{N})\neq l_1.$ By the assumption, there exists an element $x=\mu(A)\in E$ such that $m^{-1}\sigma_mx\not\to0$ in $E.$ Let $F(0,\infty)$ be a symmetric function space constructed in Proposition \ref{func-seq}. Since $E(\mathbb{N})\neq l_1,$ it follows that $F(0,\infty)\not\subset L_1(0,\infty).$ Recall that the space $E(\mathbb{N})$ is naturally embedded into the space $F(0,\infty)$ and that the norms $\|\cdot\|_E$ and $\|\cdot\|_F$ are equivalent on $E(\mathbb{N}).$ We have $x\in F$ and $m^{-1}\sigma_mx\not\to 0$ in $F(0,\infty).$ By Theorem \ref{sim exists}, there exists a positive symmetric functional $0\leq\varphi\in F(0,\infty)^*.$ The restriction of the functional $\varphi$ to $E(\mathbb{N})$ is a nontrivial positive symmetric functional on $E(\mathbb{N}).$ Let $\mathcal{L}(\varphi)$ be a functional on $E(\mathcal{M},\tau)$ defined in Theorem \ref{lifting theorem}. Clearly, $\mathcal{L}(\varphi)$ is a nontrivial positive symmetric functional on $E(\mathcal{M},\tau).$

$\eqref{fourth main condition}\Rightarrow\eqref{second main condition}$ The proof is very similar to that of the implication $\eqref{fourth main condition}\Rightarrow\eqref{first main condition}$ and is, therefore, omitted. The only difference is that references to Theorem \ref{sim sing exists fin}, Theorem \ref{sim sing exists inf} or Theorem \ref{sim exists} have to be replaced with references to Theorem \ref{fs sing exists fin}, Theorem \ref{fs sing exists inf} or Theorem \ref{fs exists}, respectively.

$\eqref{fourth main condition}\Rightarrow\eqref{third main condition}$ Firstly, we assume that the algebra $\mathcal{M}$ is finite. Without loss of generality, $\tau(1)=1.$ Let $E(\mathcal{M},\tau)$ be a symmetric operator space and let $E(0,1)$ be the corresponding symmetric function space. By the assumption, there exists an element $x=\mu(A)\in E(0,1)$ such that $m^{-1}\sigma_mx\not\to0$ in $E(0,1).$ By Theorem \ref{nonintegrable theorem}, there exists a positive symmetric but not fully symmetric functional $\varphi\in E(0,1)^*.$ Let $\mathcal{L}(\varphi)$ be a functional on $E(\mathcal{M},\tau)$ defined in Theorem \ref{lifting theorem}. Clearly, $\mathcal{L}(\varphi)$ is a symmetric but not fully symmetric functional on $E(\mathcal{M},\tau).$

The case when $\mathcal{M}$ is an infinite atomless von Neumann algebra can be treated in a similar manner. The only difference is that the reference to Theorem \ref{nonintegrable theorem} has to be replaced with the reference to either Theorem \ref{nonintegrable theorem} or Theorem \ref{integrable theorem}.

Let $E(\mathcal{M},\tau)$ be a symmetric operator space on a atomic von Neumann algebra $\mathcal{M}$ and let $E(\mathbb{N})$ be the corresponding symmetric sequence space. It follows from the assumption that $E(\mathcal{M},\tau)\neq L_1(\mathcal{M},\tau)$ or, equivalently, $E(\mathbb{N})\neq l_1.$ By the assumption, there exists an element $x=\mu(A)\in E$ such that $m^{-1}\sigma_mx\not\to0$ in $E.$ Let $F(0,\infty)$ be a symmetric function space constructed in Proposition \ref{func-seq}. Since $E(\mathbb{N})\neq l_1,$ it follows that $F(0,\infty)\not\subset L_1(0,\infty).$ Recall that the space $E(\mathbb{N})$ is naturally embedded into the space $F(0,\infty)$ and that the norms $\|\cdot\|_E$ and $\|\cdot\|_F$ are equivalent on $E(\mathbb{N}).$ We have $x\in F$ and $m^{-1}\sigma_mx\not\to 0$ in $F(0,\infty).$ By Theorem \ref{nonintegrable theorem}, there exists a positive symmetric functional $\varphi\in F(0,\infty)^*$ and a function $0\leq y\prec\prec x$ such that $\varphi(y)>\varphi(x).$ Set $z=\mathbf{E}(\mu(y)|\{(n-1,n)\}_{n\in\mathbb{N}}).$ Clearly, $z\in E(\mathbb{N})$ and $\varphi(z)=\varphi(y)>\varphi(x).$ Hence, the restriction of the functional $\varphi$ to $E(\mathbb{N})$ is a positive symmetric but not fully symmetric functional on $E(\mathbb{N}).$ Let $\mathcal{L}(\varphi)$ be a functional on $E(\mathcal{M},\tau)$ defined in Theorem \ref{lifting theorem}. Clearly, $\mathcal{L}(\varphi)$ is a positive symmetric but not fully symmetric functional on $E(\mathcal{M},\tau).$ 
\end{proof}

\section{Appendix}\label{app}

In this appendix, we set $\mathcal{A}=\{(n-1,n)\}_{n\in\mathbb{N}}.$

\begin{lem}\label{pave maj} If $x,y\in(L_1+L_{\infty})(0,\infty)$ are positive functions, then
$$\mathbf{E}(\mu(x+y)|\mathcal{A})\lhd \mathbf{E}(\mu(x)|\mathcal{A})+\mathbf{E}(\mu(y)|\mathcal{A})\lhd 2\sigma_{1/2}\mathbf{E}(\mu(x+y)|\mathcal{A}).$$
\end{lem}
\begin{proof} Recall that
$$\mu(x+y)\prec\prec\mu(x)+\mu(y)\prec\prec2\sigma_{1/2}\mu(x+y).$$
It follows that
$$\int_0^b\mu(s,x+y)ds\leq\int_0^b(\mu(s,x)+\mu(s,y))ds,$$
$$\int_0^{2a}\mu(s,x+y)ds\geq\int_0^a(\mu(s,x)+\mu(s,y))ds.$$
Let now $a,b$ be positive integers. Subtracting the above inequalities, we obtain
$$\int_{2a}^b\mathbf{E}(\mu(x+y)|\mathcal{A})(s)ds=\int_{2a}^b\mu(s,x+y)ds\leq$$
$$\leq\int_a^b(\mu(s,x)+\mu(s,y))ds=\int_a^b\mathbf{E}(\mu(x)+\mu(y)|\mathcal{A})(s)ds.$$
Similarly, we have
$$\int_{2a}^b\mathbf{E}(\mu(x)+\mu(y)|\mathcal{A})(s)ds\leq\int_{2a}^{2b}\mathbf{E}(\mu(x+y)|\mathcal{A})(s)ds.$$

\end{proof}

\begin{cor} The quasi-norm in Construction \ref{func-seq} is a norm.
\end{cor}
\begin{proof} It follows from Lemma \ref{pave maj} that
$$\mathbf{E}(\mu(x+y)|\mathcal{A})\lhd \mathbf{E}(\mu(x)+\mu(y)|\mathcal{A})$$
provided that $x,y$ are positive functions. By Theorem \ref{ks majorization theorem},
$$\|\mathbf{E}(\mu(x+y)|\mathcal{A})\|_E\leq\|\mathbf{E}(\mu(x)|\mathcal{A})\|_E+\|\mathbf{E}(\mu(y)|\mathcal{A})\|_E.$$

\end{proof}

\begin{lem}\label{stupid estimate} Let $y=\mu(y)\in(L_1+L_{\infty})(0,\infty).$ It follows that
$$\int_{2^{-k}\lambda a}^by(s)ds\leq\frac{\lambda}{\lambda-1}\int_a^b(\sigma_{2^k}y)(s)ds$$
provided that $b\geq\lambda a.$
\end{lem}
\begin{proof} Let $\alpha$ be the average value of $y$ on the interval $[2^{-k}\lambda a,2^{-k}b].$ Clearly, $y\leq\alpha$ on the interval $[2^{-k}\lambda a,b]$ and $y\geq\alpha$ on the interval $[2^{-k}a,2^{-k}b].$ Thus, $\sigma_{2^k}y\geq\alpha$ on the interval $[a,b].$ Therefore,
$$\int_{2^{-k}\lambda a}^by(s)ds\leq(b-2^{-k}\lambda a)\alpha\leq\frac{\lambda}{\lambda-1}(b-a)\alpha\leq \frac{\lambda}{\lambda-1}\int_a^b(\sigma_{2^k}y)(s)ds.$$

\end{proof}

\begin{thm} If $\{x_n\}_{n\in\mathbb{N}}$ be a Cauchy sequence in $F,$ then there exists $x\in F$ such that $x_n\to x$ in $F.$
\end{thm}
\begin{proof} For every $k>0,$ there exists $m_k$ such that $\|x_m-x_{m_k}\|_F\leq 4^{-k}$ for $m\geq m_k.$ Set $y_k=x_{m_{k+1}}-x_{m_k}.$ Clearly, $\|y_k\|_F\leq 4^{-k}$ for every $k\in\mathbb{N}.$ In particular, the series $\sum_{k=1}^{\infty}y_k$ converges in $L_{\infty}(0,\infty).$

Set $z_n=\sum_{k=n}^{\infty}\sigma_{2^k}\mu(y_k).$ We claim that $z_n\in F$ and $z_n\to0$ in $F.$ Indeed,
$$\mu(y_k)\leq\|y_k\|_{\infty}\chi_{(0,1)}+T\mathbf{E}(\mu(y_k)|\mathcal{A}).$$
Here, $T$ is a shift to the right. It follows that
$$\mathbf{E}(\mu(z_n)|\mathcal{A})\leq\sum_{k=n}^{\infty}\sigma_{2^k}(\|y_k\|_{\infty}\chi_{(0,1)}+T\mathbf{E}(\mu(y_k)|\mathcal{A})).$$
Therefore,
$$\|z_n\|_F\leq\|z_n\|_{\infty}+\sum_{k=n}^{\infty}2^k\|\|y_k\|_{\infty}\chi_{(0,1)}+T\mathbf{E}(\mu(y_k)|\mathcal{A})\|_E\leq$$
$$\leq\|z_n\|_{\infty}+\sum_{k=n}^{\infty}2^{k+1}\|y_k\|_F\leq\frac13\cdot 4^{1-n}+2^{2-n}=o(1).$$
It follows from Lemma 8.5 of \cite{KS} that
$$\int_{\lambda a}^b\mu(s,\sum_{k=n}^{\infty}y_k)ds\leq\sum_{k=n}^{\infty}\int_{2^{-k}\lambda a}^b\mu(s,y_k)ds.$$
It follows from Lemma \ref{stupid estimate} that
$$\int_{2^{-k}\lambda a}^b\mu(s,y_k)ds\leq\frac{\lambda}{\lambda-1}\int_a^b(\sigma_{2^k}\mu(y_k))(s)ds.$$
Therefore,
$$\int_{\lambda a}^b\mu(s,\sum_{k=n}^{\infty}y_k)ds\leq\frac{\lambda}{\lambda-1}\int_a^bz_n(s)ds.$$
Hence,
$$\sum_{k=n}^{\infty}y_k\lhd\frac{\lambda}{\lambda-1}z_n.$$
Since $\lambda>1$ is arbitrarily large, it follows from Theorem \ref{ks majorization theorem} that
$$\|\sum_{k=n}^{\infty}y_k\|_F\leq\|z_n\|_F\to0.$$
Thus, the series $\sum_{k=1}^{\infty}y_k$ does converge in $F.$ The assertion follows immediately.

\end{proof}

\section{Proof of Figiel-Kalton theorem}\label{fkt}

The proof of Theorem \ref{fk theorem} follows from the combinations of Lemmas below.

\begin{lem} Let $E$ be a symmetric Banach space either on the interval $(0,1)$ or on the semi-axis. If $x\in Z_E,$ then $C(\mu(x_+)-\mu(x_-))\in E.$\
\end{lem}
\begin{proof} Let $x=\sum_{k=1}^n(x_k-y_k)$ with $x_k,y_k\in E_+$ and $\mu(x_k)=\mu(y_k),$ $1\leq k\leq n.$ Set
$$z=x_++\sum_{k=1}^ny_k=x_-+\sum_{k=1}^nx_k.$$
It follows from the definition of $C$ and \eqref{rear sum estimate} that
$$C\mu(z)\leq C(x_+)+\sum_{k=1}^nC\mu(y_k)=C(\mu(x_+)-\mu(x_-))+C\mu(x_-)+\sum_{k=1}^nC\mu(x_k).$$
Using the second inequality in \eqref{rear sum estimate}, we obtain
$$\int_0^t(\mu(s,x_-)+\sum_{k=1}^n\mu(s,x_k))ds\leq\int_0^{(n+1)t}\mu(s,z)ds\leq\int_0^t\mu(s,z)ds+nt\mu(t,z).$$
Therefore,
$$C\mu(z)\leq C\mu(z)+C(\mu(x_+)-\mu(x_-))+n\mu(z).$$
It follows that $C(\mu(x_-)-\mu(x_+))\leq n\mu(z).$ Similarly, $C(\mu(x_+)-\mu(x_-))\leq n\mu(z)$ and the assertion follows. 
\end{proof}

\begin{lem} Let $E$ be a symmetric Banach space either on the interval $(0,1)$ or on the semi-axis. If $x\in\mathcal{D}_E,$ then $C(\mu(x_+)-\mu(x_-))\in Cx+E.$
\end{lem}
\begin{proof} Since $x\in\mathcal{D}_E,$ it follows that $x=\mu(a)-\mu(b)$ with $a,b\in E.$ Set $u=\mu(a)-x_+\geq0.$ Clearly, $\mu(a)=u+x_+$ and $\mu(b)=u+x_-.$ It follows from the definition of $C$ and \eqref{rear sum estimate} that
$$C\mu(a)\leq C\mu(u)+C\mu(x_+)=C(\mu(x_+)-\mu(x_-))+C\mu(u)+C\mu(x_-).$$
Using the second inequality in \eqref{rear sum estimate}, we obtain
$$C\mu(x_-)+C\mu(u)\leq C\mu(b)+\mu(b).$$
It follows that
$$Cx\leq C(\mu(x_+)-\mu(x_-))+\mu(b).$$
Similarly,
$$Cx\geq C(\mu(x_+)-\mu(x_-))-\mu(a)$$
and the assertion follows. \end{proof}

\begin{lem} Let $E=E(0,\infty)$ be a symmetric space on the semi-axis. If $x\in\mathcal{D}_E$ is such that $Cx\in E,$ then $x\in Z_E.$
\end{lem}
\begin{proof} Define a partition $\mathcal{A}=\{(2^n,2^{n+1})\}_{n\in\mathbb{Z}}$ and set $x_1=\mathbf{E}(x|\mathcal{A}).$ If $x=\mu(a)-\mu(b)$ with $a,b\in E,$ then $x_1=\mathbf{E}(\mu(a)|\mathcal{A})-\mathbf{E}(\mu(b)|\mathcal{A}).$ Clearly,
$$\mathbf{E}(\mu(a)|\mathcal{A})\leq\sigma_2\mu(a)\in E,\quad \mathbf{E}(\mu(b)|\mathcal{A})\leq\sigma_2\mu(b)\in E$$
are decreasing functions. It follows that $x_1\in\mathcal{D}_E.$ It is easy to see that
$$|Cx_1-Cx|\leq2\sigma_2(\mu(a)+\mu(b)).$$
Therefore, $Cx_1\in E.$ Define a function $z\in E$ by setting
$$z(t)=(Cx_1)(2^{n+1}),\quad t\in(2^n,2^{n+1}).$$
Clearly, $x_1=2z-\sigma_2z\in Z_E.$

Consider the function $x-x_1$ on the interval $(2^n,2^{n+1}).$ By Kwapien theorem \cite{Kwapien}, there exist positive equimeasurable functions $y_{1n},y_{2n}$ supported on $(2^n,2^{n+1})$ such that
$$\mu(y_{1n})=\mu(y_{2n}),\quad\|y_{1n}\|_{\infty},\|y_{2n}\|_{\infty}\leq 6\|(x-x_1)\chi_{(2^n,2^{n+1})}\|_{\infty}.$$
Set $y_1=\sum_{n\in\mathbb{N}}y_{1n}$ and $y_{2n}=\sum_{n\in\mathbb{N}}y_{2n}.$ It follows that $y_1,y_2\in E_+.$ Since $x-x_1=y_1-y_2$ and $\mu(y_1)=\mu(y_2),$ it follows that $x-x_1\in Z_E.$ The assertion follows immediately. 
\end{proof}

\end{document}